\documentclass[10pt]{amsart}

\usepackage{amssymb, dsfont} 
\usepackage{fullpage, graphicx} 
\usepackage{hyperref} 
\usepackage{enumitem} 
\usepackage{caption, subcaption, threeparttable} 
\usepackage{xcolor} 
\usepackage{thmtools} 

\hypersetup{colorlinks=true, urlcolor=purple, citecolor=blue, linkcolor=red, pdfstartview={XYZ null null 0.90}}

\DeclareMathOperator{\lcm}{lcm}
\DeclareMathOperator{\Int}{Int}

\DeclareMathOperator{\Cl}{Cl}
\DeclareMathOperator{\disc}{disc}
\DeclareMathOperator{\sign}{sign}
\DeclareMathOperator{\sg}{sg}

\DeclareMathOperator{\Tr}{Tr}
\DeclareMathOperator{\Lk}{Lk}
\DeclareMathOperator{\SL}{SL}
\DeclareMathOperator{\PSL}{PSL}
\DeclareMathOperator{\GL}{GL}
\DeclareMathOperator{\Riv}{Riv}
\newcommand{\Z}{\ensuremath{\mathbb{Z}}}
\newcommand{\R}{\ensuremath{\mathbb{R}}}
\newcommand{\Q}{\ensuremath{\mathbb{Q}}}

\newcommand{\sm}[4]{\ensuremath{\left(\begin{smallmatrix} #1 & #2\\#3 & #4\end{smallmatrix}\right)}}
\newcommand{\genmtx}{\sm{a}{b}{c}{d}}

\newtheorem{introtheorem}{Theorem}
\newtheorem{introproposition}[introtheorem]{Proposition}
\newtheorem{introcorollary}[introtheorem]{Corollary}

\newtheorem{theorem}{Theorem}[section]
\newtheorem{conjecture}[theorem]{Conjecture}
\newtheorem{corollary}[theorem]{Corollary}
\newtheorem{lemma}[theorem]{Lemma}
\newtheorem{proposition}[theorem]{Proposition}

\theoremstyle{definition}

\newtheorem{introdefinition}[introtheorem]{Definition}

\newtheorem{errata}{Post-publication remark}

\newtheorem{algorithm}[theorem]{Algorithm}
\newtheorem{definition}[theorem]{Definition}

\newtheorem{remark}[theorem]{Remark}
\newtheorem{example}[theorem]{Example}

\numberwithin{equation}{section}

\begin{document}

\title[Computing intersections]{Computing intersections of closed geodesics on the modular curve}
\author[J. Rickards]{James Rickards}
\address{McGill University, Montr\'{e}al, Qu\'{e}bec, Canada}
\email{james.rickards@mail.mcgill.ca}
\urladdr{https://www.math.mcgill.ca/rickards/}
\date{\today}
\thanks{This research was supported by an NSERC Vanier Scholarship.}
\thanks{\textbf{Appears in J. Number Theory (2021), \url{https://doi.org/10.1016/j.jnt.2020.11.024}}}
\thanks{\copyright 2021. This manuscript version is made available under the CC-BY-NC-ND 4.0 license \url{http://creativecommons.org/licenses/by-nc-nd/4.0/}}
\subjclass[2020]{Primary 11E16; Secondary 20H10}
\keywords{Indefinite binary quadratic forms, modular geodesics, Conway topograph}
\date{\today}

\begin{abstract}

In a recent work of Duke, Imamo\={g}lu, and T\'{o}th, the linking number of certain links on the space $\SL(2,\Z)\backslash\SL(2,\R)$ is investigated. This linking number has an alternative interpretation as the intersection number of closed geodesics on the modular curve, which is the focus of this paper. By relating the intersection number to a combinatorial computation involving rivers of Conway topographs, an efficient algorithm for computing intersection numbers is produced. A formula for the total intersection of a pair of positive coprime fundamental discriminants is also derived, which can be thought of as a real quadratic analogue of a classical result of Gross and Zagier. The paper ends with numerical computations and distribution questions relating to intersection numbers.
\end{abstract}
\maketitle

\setcounter{tocdepth}{1}
\tableofcontents

\section*{Introduction}

Consider the space $\SL(2,\Z)\backslash\SL(2,\R)$, which is diffeomorphic to the complement of a trefoil knot in $S^3$ (see \cite{Milnor71}). Given a hyperbolic element $\gamma\in\SL(2,\Z)$, one can define the knots $[\tilde{\gamma}_{+}]$ and $[\tilde{\gamma}_{-}]$ (see Section \ref{subseclink}), whose sum is the null-homologous link $[\tilde{\gamma}]$. In Section $3$ of \cite{Ghys07}, Ghys studies the linking number of $[\tilde{\gamma}_{\pm}]$ with the removed trefoil. His answer is expressed in terms of the Rademacher function, which is directly related to the classical Dedekind $\eta$ function. 

The linking number of two distinct links $[\tilde{\sigma}]$ and $[\tilde{\gamma}]$ is considered in \cite{DIT17}. The authors produce similar results to Ghys, by relating their answer to a modular cocyle (as opposed to a modular form). However, the final linking number formula involves somewhat complicated objects, and is not particularly amenable to explicit computation. Along the way, they use a theorem of Birkhoff in \cite{BK17}, which shows that their linking numbers correspond to intersection numbers of modular geodesics. In this paper, we study those intersection numbers, and show that they can be computed from the combinatorial data arising from Conway's topograph. 

Along the way, we demonstrate a connection between intersection numbers and the Gross-Zagier formula, as found in \cite{GZ85}. This connection is further explored in \cite{JR21shim}, where some of the results of this paper are generalized from the classical modular curve case to Shimura curves. In this generalization, analogous intersection numbers are conjecturally related to the work of Darmon and Vonk on explicit class field theory for real quadratic fields, \cite{DV20}.

\subsection*{Overview of the main results}

Let $\Gamma$ be a discrete subgroup of $\PSL(2,\R)$, and let $\mathbb{H}$ denote the upper half plane. The space $\Gamma\backslash\mathbb{H}$ can be given the structure of a Riemann surface, and we will consider the oriented curves on this surface which come from an upper half plane geodesic. Such curves are called \textit{modular geodesics}, and we will concern ourselves with closed modular geodesics. Given two such modular geodesics, we ask the question: how many times do they intersect?

In the case of $\Gamma=\PSL(2,\Z)$, the question can be rephrased in terms of quadratic forms. Let $q(x,y)=Ax^2+Bxy+Cy^2:=[A,B,C]$ be a quadratic form of discriminant $D=B^2-4AC$. The group $\PSL(2,\Z)$ acts on $q$ (on the right) via
\[\gamma\circ q(x,y):=q(ax+by,cx+dy),\text{ where }\gamma=\genmtx.\]
Write $q\sim q'$ if the quadratic forms $q,q'$ are related by an element of $\PSL(2,\Z)$. We can extend the equivalence to $n-$tuples of quadratic forms as follows:
\[(q_1,q_2,\ldots,q_n)\sim_n(q_1',q_2',\ldots,q_n')\text{ if there exists a $\gamma\in\PSL(2,\Z)$ such that }\gamma\circ q_i=q_i'\text{ for all $1\leq i\leq n$.}\]
Given a primitive indefinite binary quadratic form (abbreviated PIBQF) $q$, the \textit{reciprocal} form is $-q$, where all the coefficients are negated. Write $q\not\sim_{\pm} q'$ if $q$ is not equivalent to either $q'$ or $-q'$, and call $q,q'$ a \textit{strongly inequivalent} pair. Note that quadratic forms with distinct discriminants are strongly inequivalent, and the notion of strong inequivalence extends to pairs of equivalence classes.

Given an indefinite quadratic form $q=[A,B,C]$ of discriminant $D>0$, the equation $q(x,1)=0$ has two real solutions, the roots of $q$. Let the \textit{first root} and \textit{second root} be
\[q_f:=\dfrac{-B+\sqrt{D}}{2A},\qquad q_s:=\dfrac{-B-\sqrt{D}}{2A},\]
respectively. In particular, the second root is smaller than the first root if and only if $A>0$. The upper half plane geodesic running from $q_s$ to $q_f$ is denoted by $\ell_q$ and called the \textit{root geodesic}. Let $\gamma_q$ be the \textit{invariant automorph} of $q$ (explicitly given in Definition \ref{defautom}). Since $\gamma_q$ fixes both $q_f$ and $q_s$, $\gamma_q(\ell_q)=\ell_q$. Therefore the open curve $\ell_q$ maps in an infinite-to-one way onto a closed modular geodesic of the modular curve $\Gamma\backslash\mathbb{H}$. Denote the image by $\tilde{\ell_q}$, and observe that it only depends on the equivalence class of $q$. The modular geodesic $\tilde{\ell}_{-q}$ overlaps $\tilde{\ell}_q$, but has the opposite orientation.

\begin{introdefinition}
Let $q_1,q_2$ be a pair of PIBQFs. Their \textit{unweighted intersection number}, denoted by $\Int(q_1,q_2)$, is defined to be
\[|\tilde{\ell}_{q_1}\pitchfork\tilde{\ell}_{q_2}|,\]
the number of transverse intersections of $\tilde{\ell}_{q_1}$ with $\tilde{\ell}_{q_2}$. Note that we are taking the set-theoretic intersection of the curves, and not the homological intersection (which is always $0$ due to the genus of the modular curve being $0$). Furthermore, $\Int(q_1, q_2)$ is well-defined, finite, non-negative, and symmetric.
\end{introdefinition}

We will present two approaches to understanding the behaviour of intersection numbers. The first approach is very algebraic in nature, and lends itself to theoretical counts of intersection numbers. The second approach is combinatorial, and yields an efficient practical algorithm. For the algebraic approach, we first present an alternate characterization of $\Int(q_1,q_2)$.

\begin{introproposition}\label{mainprop1}
The unweighted intersection number of PIBQFs $q_1,q_2$ of discriminants $D_1,D_2$ is the size of the set
\[\left\{(q_1',q_2'):q_1\sim q_1', q_2\sim q_2', |\ell_{q_1'}\pitchfork\ell_{q_2'}|=1\right\}/\sim_2.\]
Furthermore,
\[|\ell_{q_1'}\pitchfork\ell_{q_2'}|=1\Leftrightarrow |B_{\Delta}(q_1',q_2')|<\sqrt{D_1D_2},\]
where $B_{\Delta}([A_1,B_1,C_1],[A_2,B_2,C_2]):=B_1B_2-2A_1C_2-2A_2C_1$.
\end{introproposition}
In particular, 
\[\Int(q_1,q_2)=\left|\left\{(q_1',q_2'):q_1\sim q_1', q_2\sim q_2',|B_{\Delta}(q_1',q_2')|<\sqrt{D_1D_2}\right\}/\sim_2\right|.\]

Let $B_{\Delta}(q_1', q_2')=x$, and then the inequality $x^2<D_1D_2$ pops out of Proposition \ref{mainprop1}. This inequality (along with $B_{\Delta}$) is reminiscent of the results of Gross and Zagier on the factorization of difference of $j-$values, found in \cite{GZ85}. Further connections come from Theorem \ref{mainthm1}, which is the real quadratic analogue of Proposition $6.1$ of \cite{GZ85}. Let $D_1,D_2$ be distinct positive discriminants, and
\begin{itemize}
\item Define $\Cl^+(D_i)$ to be the narrow class group of discriminant $D_i$, the set of equivalence classes of primitive quadratic forms of discriminant $D_i$. The size of $\Cl^+(D_i)$ is $h^+(D_i)$, the narrow class number.
\item Let $K=\Q (\sqrt{D_1D_2})$, let $L=\Q (\sqrt{D_1},\sqrt{D_2})$, and define $r_{L/K}(\mathfrak{a})$ to be the number of integral ideals $\mathfrak{A}$ of $L$ for which $N_{L/K}(\mathfrak{A})=\mathfrak{a}$. 
\item Define 
\[S_{D_1,D_2}:=\{n: |n|<\sqrt{D_1D_2}\text{ and }n\equiv D_1D_2\pmod{2}\}.\]
\item For $n\in S_{D_1,D_2}$, define
\[p(n):=\left|\left\{(q_1,q_2):\disc(q_1)=D_1, \disc(q_2)=D_2, B_{\Delta}(q_1,q_2)=n\right\}\diagup\sim_2\right|.\]
\item For $D_1,D_2$ coprime and fundamental, if $p$ is a prime with $\left(\frac{D_1D_2}{p}\right)\neq -1$, define $\epsilon(p)$ be the non-zero value in the set $\left\{\left(\frac{D_1}{p}\right),\left(\frac{D_2}{p}\right)\right\}$. Extend the definition of $\epsilon$ multiplicatively. Note that $\epsilon(p)$ is defined for all prime divisors $p$ of $\frac{D_1D_2-n^2}{4}$ with $n\equiv D_1D_2\pmod{2}$.
\end{itemize}

\begin{introtheorem}\label{mainthm1}
Let $D_1,D_2$ be positive coprime fundamental discriminants, let $n\in S_{D_1,D_2}$, and factorize
\[\dfrac{D_1D_2-n^2}{4}=\prod_{i=1}^{r}p_i^{2e_i+1}\prod_{i=1}^s q_i^{2f_i}\prod_{i=1}^t w_i^{g_i},\]
where the $p_i$ are the primes for which $\epsilon(p_i)=-1$ that appear to an odd power, $q_i$ are the primes for which $\epsilon(q_i)=-1$ that appear to an even power, and $w_i$ are the primes for which $\epsilon(w_i)=1$. Then $p(n)=0$ if and only if $r>0$, and when $r=0$, 
\[p(n)=2\prod_{i=1}^t(g_i+1)=2\sum_{d\mid\frac{D_1D_2-n^2}{4}}\epsilon(d)=2r_{L/K}\left(\left\langle\frac{\sqrt{D_1D_2}-n}{2}\right\rangle\right).\]
The second and third expressions for $p(n)$ are also valid when $r>0$.
\end{introtheorem}

The proof in \cite{GZ85} does not automatically generalize to our situation. With motivation coming from quaternion algebras, we will sketch a proof of this result (the proof of a more general result appears in \cite{JR21shim}).

Combining Proposition \ref{mainprop1} and Theorem \ref{mainthm1} provides an expression for the total intersection number of forms of positive discriminants $D_1, D_2$, defined as 
\[\Int(D_1,D_2):=\sum_{[q_i]\in\Cl^{+}(D_i)}\Int(q_1,q_2).\]

\begin{introcorollary}\label{maincor1}
Let $D_1, D_2$ be positive coprime fundamental discriminants. Then
\[\Int(D_1,D_2)=\sum_{n\in S_{D_1,D_2}}p(n)=2\sum_{n\in S_{D_1,D_2}}\sum_{d\mid\frac{D_1D_2-n^2}{4}}\epsilon(d).\]
\end{introcorollary}

While Corollary \ref{maincor1} allows the computation of the average intersection number for a pair of discriminants, it has a hard time accessing the individual intersection numbers that form its components. For this, we introduce the combinatorial approach.

Let the topographs of PIBQFs $q_1, q_2$ be $T_1, T_2$ respectively. The $T_i$ are $3-$regular infinite connected trees drawn in the plane, each equipped with a river path $R_i$. The path $R_i$ is an infinite periodic path, which corresponds to the set of forms $[A,B,C]$ similar to $q_i$ with $AC<0$. The periodic part refers to the periodicity of the sequence of left and right branches taken by the path (left and right are canonical as the graph is embedded into the plane). By picking an oriented edge in $T_1$ and $T_2$, there is a natural way to superimpose the topographs on top of each other, which allows one to consider the behaviours of the rivers $R_1$, $R_2$.

\begin{introtheorem}\label{mainthm2}
The intersection number $\Int(q_1, q_2)$ is equal to the number of ways to superimpose the topographs $T_1, T_2$ onto each other so that the river paths $R_1$ and $R_2$ meet and cross, modulo the periods of the rivers.
\end{introtheorem}

Computing the river paths $R_i$ can be done directly from the continued fraction expansion of the first roots of $q_i$, and the computation of the intersection number is a combinatorial function of the river paths.

In Section \ref{sechyper}, we introduce the intersection number in the general case. We give multiple interpretations of the intersection number, and also give a description of the intersection points and angles. In Section \ref{secpsl2z}, we apply the theory to the case of $\Gamma=\PSL (2,\Z)$, and prove Proposition \ref{mainprop1} and Theorem \ref{mainthm1}. In Section \ref{secconway}, we introduce the Conway topograph, and use it to prove Theorem \ref{mainthm2}. In Section \ref{secnumcalc}, we provide some explicit examples, and use our algorithm to produce interesting data on the distribution of intersection numbers.

\subsection*{Acknowledgments}
I would like to thank my advisor Henri Darmon for suggesting this problem and providing excellent feedback along the way. I would also like to thank Jan Vonk for introducing me to the Conway topograph, and engaging in many discussions about this work. I am grateful to the anonymous reviewer who provided many helpful comments and suggestions.

\section{Hyperbolic geometry and intersection numbers}\label{sechyper}

Let $\overline{\mathbb{H}}:=\mathbb{H}\cup\R\cup i\infty$ be the upper half plane with its boundary. For $z_1,z_2\in\overline{\mathbb{H}}$, the geodesic segment connecting $z_1,z_2$ is either a vertical line segment between $z_1$ and $z_2$, or the segment between $z_1$ and $z_2$ of the unique circle with centre on the real line which passes through $z_1$ and $z_2$. Denote this segment by $\ell_{z_1,z_2}$, where we do not include the endpoints $z_1,z_2$. We think of the geodesic as running from $z_1$ to $z_2$, and refer to this notion as the orientation of the geodesic. Define $\dot{\ell}_{z_1,z_2}$ to mean $\ell_{z_1,z_2}\cup\{z_1\}$.

Recall that M\"{o}bius maps act on $\mathbb{H}$ and $\overline{\mathbb{H}}$, and they take geodesic segments to geodesic segments. In particular, if $\gamma\in\SL (2,\Z)$ and $z_1,z_2\in\overline{\mathbb{H}}$, then we have
\[\gamma(\ell_{z_1,z_2})=\ell_{\gamma z_1,\gamma z_2}.\]

When working explicitly with hyperbolic matrices in $\PSL(2,\R)$, we need to lift them to $\SL(2,\R)$. By convention, we will always take the lift with positive trace.

\subsection{Roots of hyperbolic matrices}\label{secrtshyper}
Let $\gamma=\genmtx\in\PSL(2,\R)$ be a hyperbolic matrix, so that the equation $\gamma(x)=x$ has two distinct real roots. We label one root to be the first (attracting) root $\gamma_f$, and the other to be the second (repelling) root $\gamma_s$, via the equations
\[\lim_{n\rightarrow\infty}\gamma^n(x)=\gamma_f,\qquad\lim_{n\rightarrow\infty}\gamma^{-n}(x)=\gamma_s,\]
for any $x\in\mathbb{P}^1(\R)$ that is not a root of $\gamma$. In particular, $\gamma^{-1}$ has the same roots as $\gamma$, but with the first and second roots swapped. The first root can be algebraically expressed as
\begin{align*}
\gamma_f=\begin{cases}
\dfrac{a-d+\sqrt{(a+d)^2-4}}{2c}, & \text{if $c\neq 0$;}\\
\infty,  &\text{if $c=0$ and $a>1$;}\\ 
\dfrac{b}{d-a}, &\text{if $c=0$ and $a<1$.}\end{cases}
\end{align*}

Note that if $\sigma\in\PSL(2,\R)$, then the roots of $\sigma\gamma\sigma^{-1}$ are $\sigma(\gamma_f),\sigma(\gamma_s)$. Furthermore, it follows that the ``firstness'' is preserved, i.e. 
\[(\sigma\gamma\sigma^{-1})_f=\sigma(\gamma_f),\qquad(\sigma\gamma\sigma^{-1})_s=\sigma(\gamma_s).\]

We end this section with a definition of the sign of an intersection. It is not completely canonical, as one could choose the opposite sign everywhere.

\begin{definition}\label{defintsign}
Let $y_1,y_2,z_1,z_2\in\mathbb{P}_1(\R)$ be such that $\ell_1=\ell_{y_1,y_2}$ and $\ell_2=\ell_{z_1,z_2}$ are distinct geodesics intersecting in the upper half plane. Travel along $\ell_1$ from $y_1$ to $y_2$, and consider which side $z_1$ lies on. If it is on the right hand side of $\ell_1$, then the sign of the \textit{ordered} intersection of $\ell_1,\ell_2$, denoted $\sg (\ell_1,\ell_2)$, is $+1$. Otherwise, the sign is $-1$.
\end{definition}

\begin{proposition}
Let $\ell_1,\ell_2$ be geodesics that intersect in a unique point in the upper half plane, and let $\ell_1^{-1}$ denote the geodesic $\ell_1$ run backwards. Then
\[\sg (\ell_1,\ell_2)=-\sg (\ell_2,\ell_1)=-\sg (\ell_1^{-1},\ell_2),\]
i.e. swapping the order of the inputs or travelling along one of the geodesics backwards negates the sign. Furthermore, if $\gamma\in\SL (2,\R)$, then
\[\sg (\ell_1,\ell_2)=\sg (\gamma\ell_1,\gamma\ell_2).\]
\end{proposition}
\begin{proof}
The first result is immediate from the definition of the sign. The second half only needs to be checked for a set of generators of $\SL(2,\R)$. A suitable set of generators is $\sm{0}{1}{-1}{0}\cup\left\{\sm{1}{x}{0}{1}:x\in\R\right\}$, and this is a straightforward check.
\end{proof}

\subsection{The intersection number}
Let $\Gamma$ be a discrete subgroup of $\PSL(2,\R)$. If $z_1,z_2\in\mathbb{P}_1(\R)$, then $\ell_{z_1,z_2}$ generally does not project to a closed curve in $\Gamma\backslash\mathbb{H}$. To get closed curves, take $\gamma\in\Gamma$ to be hyperbolic, and define
\[\ell_{\gamma}:=\ell_{\gamma_s,\gamma_f},\]
the root geodesic corresponding to the fixed points of $\gamma$. Note that this geodesic is preserved by $\gamma$, and therefore when we descend to $\Gamma\backslash\mathbb{H}$ this descends to $\tilde{\ell_{\gamma}}$, a closed curve. 

For any $\sigma\in\Gamma$, 
\[\ell_{\sigma\gamma\sigma^{-1}}=\ell_{\sigma(\gamma_s),\sigma(\gamma_f)}=\sigma(\ell_{\gamma}),\]
whence
\[\tilde{\ell}_{\sigma\gamma\sigma^{-1}}=\tilde{\ell}_{\gamma}.\]
Therefore, geodesics coming from a hyperbolic conjugacy class of $\Gamma$ all descend to the same closed modular geodesic in $\Gamma\backslash\mathbb{H}$.

\begin{remark}\label{rmkgeod}
We will be taking $\gamma\in\Gamma$ to be \textit{primitive}, i.e. $\gamma\neq\sigma^n$ for any $\sigma\in\Gamma$, $n\geq 2$. This will ensure that
\[\left\{\sigma\in\Gamma:\sigma(\gamma_f)=\gamma_f, \text{ and }\sigma(\gamma_s)=\gamma_s\right\}=\left\{\sigma\in\Gamma:\sigma(\ell_{\gamma})=\ell_{\gamma}\right\}=\gamma^{\Z}.\]
In particular, points on the curve $\tilde{\ell_{\gamma}}$ are in bijection with points on $\dot{\ell}_{Q,\gamma(Q)}$ for any $Q\in\ell_{\gamma}$.
\end{remark}

\begin{definition}
The pair $\gamma_1,\gamma_2\in\Gamma$ is called a \textit{strongly inequivalent} pair if $\gamma_1$ is not conjugate to either $\gamma_2$ or $\gamma_2^{-1}$ in $\Gamma$. This definition extends to pairs of $\Gamma$-conjugacy classes of matrices.
\end{definition}

Note that if $\gamma_1$ and $\gamma_2$ are not strongly inequivalent, and there is a point in space that is passed through $n$ times by their closed geodesic, then this contributes $n(n-1)$ transverse intersection points.

\begin{definition}\label{defintnum}
Given a pair of primitive hyperbolic matrices $\gamma_1,\gamma_2\in\Gamma$ and a function $f$, their weighted intersection number is defined to be
\[\Int_{\Gamma}^f(\gamma_1,\gamma_2):=\sum_{p\in\tilde{\ell}_{\gamma_1}\pitchfork\tilde{\ell}_{\gamma_2}}f(\tilde{\ell}_{\gamma_1},\tilde{\ell}_{\gamma_2},p).\]
\end{definition}

The most natural choices of $f$ are $f=1$, the \textit{unweighted intersection number}, and $f$ equals the sign of the intersection, the \textit{signed intersection number}. These choices are denoted by $\Int(\gamma_1,\gamma_2)$ and $\Int^{\pm}(\gamma_1,\gamma_2)$ respectively. The advantage of the signed intersection number is it is now well defined in homology of the surface, whereas the unsigned depends on the actual geodesics. However, when the genus of $\Gamma\backslash\mathbb{H}$ is $0$ (for example $\Gamma=\PSL(2,\Z)$, the main object of study in this paper), this means that the signed intersection number is always zero!

\subsection{Alternative interpretations of the intersection number}\label{subsecalternate}
When working with intersection numbers, removing the need for the quotient space $\Gamma\backslash\mathbb{H}$ makes matters more tractable.

Let $\gamma_1,\gamma_2$ be a pair of primitive hyperbolic matrices. Pick any $z\in\ell_{\gamma_2}$, and the curve $\tilde{\ell}_{\gamma_2}$ lifts uniquely and bijects with $\dot{\ell}_{z,\gamma_2(z)}$. Each intersection point will lift to a unique pair $(\ell,P)$, where $\ell$ is $\Gamma-$conjugate to $\ell_{\gamma_1}$, and $P$ is the intersection of $\ell$ with $\dot{\ell}_{z,\gamma_2(z)}$. This is formalized in the following proposition.

\begin{proposition}\label{propfourthintersection}
Let $\gamma_1,\gamma_2,z$ be as above. Then
\[\Int_{\Gamma}^f(\gamma_1,\gamma_2)=\sum_{\substack{\gamma\sim\gamma_1\\ |\ell_{\gamma}\pitchfork\dot{\ell}_{z,\gamma_2z}|=1}}f(\gamma,\gamma_2).\]
\end{proposition}

Let $\Gamma_i:=\gamma_i^{\Z}$ be the automorph group of $\ell_{\gamma_i}$ inside $\Gamma$ for $i=1,2$. Instead of the condition that $\gamma$ is conjugate to $\gamma_1$, we could set $\gamma=\sigma\gamma_1\sigma^{-1}$ for a unique $\sigma\in\Gamma/\Gamma_1$. Similarly, the intersection point lying on $\dot{\ell}_{z,\gamma_2z}$ can be lifted to $\ell_{\gamma_2}$ by passing to the double coset $\sigma\in\Gamma_2\backslash\Gamma/\Gamma_1$. This gives us the next interpretation.

\begin{proposition}[Double coset interpretation]\label{propdoublecoset}
Let $\gamma_1,\gamma_2\in\Gamma$ be a pair of primitive hyperbolic matrices. Then
\[\Int_{\Gamma}^f(\gamma_1,\gamma_2)=\sum_{\substack{\tilde{\sigma}\in\Gamma_2\backslash\Gamma/\Gamma_1\\ |\ell_{\sigma\gamma_1\sigma^{-1}}\pitchfork\ell_{\gamma_2}|=1}}f(\sigma\gamma_1\sigma^{-1},\gamma_2).\]
\end{proposition}

A way to rephrase the above proposition is we are looking for intersecting root geodesics of conjugates of $\gamma_1,\gamma_2$ modulo the automorphs. A cleaner interpretation is the following proposition.

\begin{proposition}\label{propsimultaneouspairinterpretation}
Let $\mathcal{C}_1,\mathcal{C}_2$ be a pair of primitive hyperbolic $\Gamma-$conjugacy classes. Define an equivalence relation on $\mathcal{C}_1\times\mathcal{C}_2$ via $(\sigma_1,\sigma_2)\sim(\alpha\sigma_1\alpha^{-1},\alpha\sigma_2\alpha^{-1})$ for all $\alpha\in\Gamma$. Then
\[\Int_{\Gamma}^f(\mathcal{C}_1,\mathcal{C}_2)=\sum_{\substack{(\sigma_1,\sigma_2)\in(\mathcal{C}_1\times\mathcal{C}_2)/\sim \\ |\ell_{\sigma_1}\pitchfork\ell_{\sigma_2}|=1}}f(\sigma_1,\sigma_2).\]
\end{proposition}

\subsection{Intersection point and angle}
By lifting an intersection point into the upper half plane, we get a $\PSL(2,\Z)$-equivalence class of points. Furthermore, since M\"{o}bius maps preserve angles, the intersection point corresponds to a unique angle. This motivates studying the intersection point and angle of pairs of $\SL(2,\R)$ matrices.

\begin{definition}
For any matrix $M\in \SL (2,\R)$, define $Z_M=M-\frac{\text{Tr}(M)}{2}\text{Id}$ to be the unique matrix of trace $0$ related to $M$ by a multiple of the identity matrix. Note that this descends to $M\in\PSL (2,\R)$.
\end{definition}

If $q$ is a PIBQF, note that $Z_{\gamma_q}$ (where $\gamma_q$ is the invariant automorph, as will be defined in Definition \ref{defautom}) is a scalar multiple of a simple matrix with small entries, whereas $\gamma_q$ is not (especially if Pell's equation has a large solution).

\begin{proposition}\label{propposdet}
Let $M_1,M_2\in \SL(2,\R)$ be hyperbolic matrices with corresponding root geodesics $\ell_1,\ell_2$, and let $Z_{M_i}=Z_i$ for $i=1,2$. Then
\begin{enumerate}[label=(\roman*)]
\item $\ell_1,\ell_2$ intersect uniquely in the upper half plane if and only if
\[\det(M_1M_2-M_2M_1)>0.\]
\item In all cases,
\[\det(M_1M_2-M_2M_1)=\det(Z_1Z_2-Z_2Z_1)=4\det(Z_1Z_2)-(\Tr(Z_1Z_2))^2.\]
\item If $\ell_1,\ell_2$ intersect uniquely in the upper half plane, then 
\begin{enumerate}
\item the sign of the intersection is given by
\[\sign ((M_1M_2-M_2M_1)_{21})=\sign ((Z_1Z_2-Z_2Z_1)_{21}).\]
\item the intersection point is the fixed point of $Z_1Z_2$ that lies in the upper half plane.
\item the intersection angle $\theta$ (measured counterclockwise from the tangent to $\ell_1$ to the tangent to $\ell_2$) satisfies 
\[\tan(\theta)=\dfrac{\sqrt{\det(Z_1Z_2-Z_2Z_1)}}{\Tr(Z_1Z_2)}.\]
\end{enumerate}
\end{enumerate}
\end{proposition}
\begin{proof}
Consider conjugating $M_1,M_2$ by some $N\in\SL(2,\R)$, i.e. do
\[M_i\rightarrow NM_iN^{-1}\text{ for }i=1,2.\]
Then $\ell_i$ is taken to $N\ell_i$ and $Z_i$ is taken to $NZ_iN^{-1}$. It follows that proving the theorem for $M_1,M_2$ is equivalent to proving it for $NM_1N^{-1},NM_2N^{-1}$, except for possibly the sign of intersection (which will be treated in due course). Therefore we can replace $M_1,M_2$ by the conjugated pair, and since $M_1,M_2$ are diagonalizable over $\R$, choose $N$ to diagonalize $M_2$. Thus it can be assumed that
\[M_1=\left(\begin{matrix}a&b\\c&d\end{matrix}\right),\qquad M_2=\left(\begin{matrix}e&0\\0&\frac{1}{e}\end{matrix}\right),\]
for real numbers $a,b,c,d,e$, with $e>1$ and $ad-bc=1$. The root geodesic corresponding to $M_2$ is $\ell_{\infty,0}$, so the root geodesic of $M_1$ intersects this if and only if the product of the roots of $M_1$ is negative. This product is $\frac{-b}{c}$, so the geodesics intersect in the upper half plane if and only if $\frac{b}{c}>0$ (which includes the hypothesis that $c\neq 0$). For the determinant, we calculate
\begin{equation}\label{eqM1M2diff}
\det(M_1M_2-M_2M_1)=\det\left(\begin{matrix}0 & b(\frac{1}{e}-e)\\c(e-\frac{1}{e}) & 0\end{matrix}\right)=bc\left(e-\dfrac{1}{e}\right)^2,
\end{equation}
Since $e\neq\pm 1$, this is positive if and only if $bc>0$, which is equivalent to $\frac{b}{c}>0$, which is the first part.

For the second part, as $\frac{\text{Tr}(M_i)}{2}\text{Id}$ is a multiple of the identity, it commutes with all matrices. Thus we see that
\begin{equation}\label{eqnM2Z}
M_1M_2-M_2M_1=Z_1Z_2-Z_2Z_1.
\end{equation}
Since $Z_i$ has trace $0$, its adjugate is $-Z_i$, and we see that
\begin{align*}\det(Z_1Z_2-Z_2Z_1)= & \dfrac{1}{2}\Tr\left((Z_1Z_2-Z_2Z_1)\text{adj}(Z_1Z_2-Z_2Z_1)\right)\\
= & 2\det(Z_1Z_2)-\dfrac{1}{2}\Tr(Z_1Z_2\text{adj}(Z_1)\text{adj}(Z_2)+Z_2Z_1\text{adj}(Z_2)\text{adj}(Z_1))\\
= & 2\det(Z_1Z_2)-\Tr((Z_1Z_2)^2)\\
= & 4\det(Z_1Z_2)-\left(\Tr(Z_1Z_2)\right)^2,
\end{align*}
which completes the second point.

From now on, we assume that the root geodesics of $M_1,M_2$ intersect. For the sign of the intersection, as $e>1$, the first root of $M_2$ is $\infty$ and the second root is $0$, and so 
\[\sg (M_1,M_2)=+1\Leftrightarrow \sg (M_2,M_1)=-1\Leftrightarrow M_{1,f}>0.\] 
Since the root geodesics intersect, one root of $M_1$ is positive and the other is negative. Thus $M_{1,f}>0$ is equivalent to the first root of $M_1$ being greater than the second root, i.e. iff $c>0$. Since
\[\sign ((M_1M_2-M_2M_1)_{21})=\sign (c(e-\frac{1}{e}))=\sign (c),\]
the result follows for the matrices $M_1,M_2$. To complete the proof for all matrices, it suffices to show that the sign of $(M_1M_2-M_2M_1)_{21}$ is constant when we conjugate $M_1,M_2$. To do this, note that $M_1M_2-M_2M_1$ has trace $0$, so we write
\[M=M_1M_2-M_2M_1=\left(\begin{matrix}A&B\\C&-A\end{matrix}\right),\]
where $-A^2-BC>0$ as the root geodesics intersect. Let $N=\sm{E}{F}{G}{H}$ be any matrix in $\SL (2,\R)$, and then
\[((NM_1N^{-1})(NM_2N^{-1})-(NM_2N^{-1})(NM_1N^{-1}))_{21}=(NMN^{-1})_{21}=CH^2+2AGH-BG^2.\]
This is a quadratic form in $G,H$ with discriminant $4A^2+4BC<0$, so it is a definite form. Thus the values it takes on pairs $(G,H)\neq (0,0)$ all have the same sign, equal to the sign of $C=M_{21}$, as claimed. Equation \ref{eqnM2Z} completes this point.

For the last two points, we do the explicit calculation. The semi-circle $\ell_1$ has equation
\begin{equation}\label{eqsemicircle}
\left(x-\frac{a-d}{2c}\right)^2+y^2=\frac{(a+d)^2-4}{4c^2},\qquad y\geq 0,
\end{equation}
and the line $\ell_2$ has equation $x=0$. Thus the intersection point is given by $(x,y)=\left(0,\sqrt{\frac{b}{c}}\right)$. We calculate that
\begin{equation}\label{eqz1z2}
Z_1Z_2=\frac{1}{4}\left(\begin{matrix}(a-d)(e-\frac{1}{e}) & 2b(\frac{1}{e}-e)\\2c(e-\frac{1}{e}) & (a-d)(e-\frac{1}{e})\end{matrix}\right),
\end{equation}
and this has fixed points $\pm\sqrt{\frac{b}{c}}i$, as desired.

For the angle, we have that $\cot(\theta)$ is the slope of the tangent to $\ell_1$ at the intersection point. The slope of tangent to the circle $(x-A)^2+y^2=R^2$ at $(x_0,y_0)$ is $\frac{A-x_0}{y_0}$, so Equation \ref{eqsemicircle} gives us
\[\tan(\theta)=\frac{1}{\cot(\theta)}=\dfrac{\sqrt{b/c}}{(a-d)/(2c)}=\dfrac{2\sqrt{bc}}{a-d}.\]
Equation \ref{eqM1M2diff} combined with Equation \ref{eqnM2Z} and Equation \ref{eqz1z2} give that
\[\dfrac{\sqrt{\det(Z_1Z_2-Z_2Z_1)}}{\Tr(Z_1Z_2)}=\dfrac{\sqrt{bc}\left(e-\frac{1}{e}\right)}{\frac{1}{2}(a-d)\left(e-\frac{1}{e}\right)}=\frac{2\sqrt{bc}}{a-d}=\tan(\theta),\]
as desired.
\end{proof}

\section{Intersection numbers for the full modular group}\label{secpsl2z}

From now on, assume that $\Gamma=\PSL (2,\Z)$. The first order of business is translating matters into the language of binary quadratic forms.

\subsection{Binary quadratic forms}\label{secbqfs}

If $q$ is a PIBQF, consider the stabilizer of the action of $\PSL(2,\Z)$ on $q$. This is an infinite cyclic group, generated by an invariant automorph.

\begin{definition}\label{defautom}
Let $q=[A,B,C]$ be a PIBQF, and define
\[\gamma_q:=\sm{\frac{T-BU}{2}}{-CU}{AU}{\frac{T+BU}{2}},\]
where $(T,U)$ are the smallest positive integer solutions to Pell's equation
\[t^2-Du^2=4.\]
Then $\gamma_q$ generates the stabilizer of $q$ in $\PSL (2,\Z)$, and we call $\gamma_q$ the invariant automorph of $q$.
\end{definition}

\begin{definition}
Let $M=\genmtx\in\SL(2,\Z)$ be a primitive hyperbolic matrix with positive trace. The equation $Mx=x$ translates to $cx^2+(d-a)x-b=0$, so let $g=\gcd(c,d-a,b)$. The PIBQF associated to $M$ is defined to be
\[q_M:=\left[\dfrac{c}{g},\dfrac{d-a}{g},\dfrac{-b}{g}\right].\]
\end{definition}
It can be checked that the operations $q\rightarrow\gamma_q$ and $M\rightarrow q_M$ are inverse operations, whence we have the bijections
\begin{align*}
\text{PIBQFs}\leftrightarrow & \text{primitive hyperbolic matrices of $\SL(2,\Z)$ with positive trace}\\
\leftrightarrow & \text{primitive hyperbolic matrices of $\PSL(2,\Z)$}.
\end{align*}
The definition of first and second roots of quadratic forms and matrices are consistent with this bijection. Furthermore, the action of $\PSL(2,\Z)$ on PIBQFs corresponds to conjugation on primitive hyperbolic matrices as follows:
\begin{equation}\label{eqnbqfmatequivcorr}
\gamma_{Mq}=M^{-1}\gamma_qM\text{ for all $M\in\PSL(2,\Z)$.}
\end{equation}
In particular, equivalence classes of PIBQFs correspond to conjugacy classes of primitive hyperbolic matrices.

Note that when we defined invariant automorph, there were two possible choices, $\gamma_q$ and $\gamma_q^{-1}$. Taking the reciprocal of a form swaps between them, i.e. $\gamma_{-q}=\gamma_{q}^{-1}$.

\begin{definition}
Define a form to be reciprocal if it is equivalent to its reciprocal form. Note that a form is reciprocal if and only if every form equivalent to it is reciprocal, whence the definition will extend to an equivalence class of forms.
\end{definition}

\subsection{Intersection numbers of binary quadratic forms}

By considering the explicit form of $\gamma_q$, it follows that the root geodesics $\ell_q$ and $\ell_{\gamma_q}$ are identical, and therefore
\[\Int(q_1,q_2)=\Int_{\PSL(2,\Z)}^{1}(\gamma_{q_1},\gamma_{q_2}).\]
In particular, the first part of Proposition \ref{mainprop1} follows from Proposition \ref{propsimultaneouspairinterpretation}. The remainder follows from Proposition \ref{propsl2zintersect}, where we specialize Proposition \ref{propposdet} to $\Gamma=\PSL(2,\Z)$, and translate matrices to quadratic forms.

\begin{proposition}\label{propsl2zintersect}
Let $q_i=[A_i,B_i,C_i]$ ($i=1,2$) be a pair of PIBQFs of discriminants $D_1,D_2$ respectively. Then
\begin{enumerate}[label=(\roman*)]
\item The root geodesics of $q_1,q_2$ intersect uniquely in the upper half plane if and only if
\[|B_{\Delta}(q_1,q_2)|<\sqrt{D_1D_2}.\]
\item If the root geodesics intersect, let $x=B_{\Delta}(q_1,q_2)$. Then
\begin{enumerate}
\item the sign of the intersection is given by
\[\sg (q_1,q_2)=\sign (B_1A_2-B_2A_1)=\sign (B_2C_1-B_1C_2).\]
\item the point of intersection is the upper half plane root of
\[[-A_1B_2+A_2B_1,-2A_1C_2+2A_2C_1,-B_1C_2+B_2C_1],\]
which is a quadratic form of discriminant $x^2-D_1D_2$.
\item the angle of intersection $\theta$ satisfies
\[\tan(\theta)=\dfrac{\sqrt{D_1D_2-x^2}}{x}.\]
\end{enumerate}
\end{enumerate}
\end{proposition}
\begin{proof}
Let $T_i^2-D_iU_i^2=4$ be the smallest solution to Pell's equation ($i=1,2$), and then
\[Z_{i}:=Z_{\gamma_{q_i}}=\dfrac{U_i}{2}\left(\begin{matrix}-B_i&-2C_i\\2A_i&B_i \end{matrix}\right).\]
The determinant is $\det(Z_i)=\frac{U_i^2}{4}(-B_i^2+4A_iC_i)=-\frac{U_i^2D_i}{4}$, and
\[Z_1Z_2=\dfrac{U_1U_2}{4}\left(\begin{matrix}B_1B_2-4A_2C_1&2B_1C_2-2B_2C_1\\-2A_1B_2+2A_2B_1&B_1B_2-4A_1C_2\end{matrix}\right).\]
Thus 
\[\Tr(Z_1Z_2)=\dfrac{U_1U_2}{2}(B_1B_2-2A_1C_2-2A_2C_1)=\dfrac{U_1U_2}{2}B_{\Delta}(q_1,q_2).\]
 Proposition \ref{propposdet} says that the root geodesics intersect if and only if $(\Tr(Z_1Z_2))^2<4\det(Z_1Z_2)$, which translates to $|B_{\Delta}(q_1,q_2)|<\sqrt{D_1D_2}$.

For the sign of the intersection, Proposition \ref{propposdet} gives
\[\sg (q_1,q_2)=\sign (Z_1Z_2-Z_2Z_1)_{21}=\sign\left((U_1U_2)(B_1A_2-A_1B_2)\right)=\sign (B_1A_2-B_2A_1).\]
The equality with $\sign (B_2C_1-B_1C_2)$ comes from applying $S$ to $q_1,q_2$; the sign of the intersection remains the same, and we start with PIBQFs $[C_i,-B_i,A_i]$ instead.

The intersection point and angle come directly from plugging in these calculations into Proposition \ref{propposdet}.
\end{proof}

\begin{remark}\label{remalternatesign}
If the root geodesics of $q_1,q_2$ intersect, then an alternate interpretation of the sign of the intersection is the sign of $q_1(q_{2,f},1)$.
\end{remark}

\begin{remark}
The discriminant of the quadratic form $q_1x+q_2y$ is
\[D_1x^2+2B_{\Delta}(q_1,q_2)xy+D_2y^2,\]
so $B_{\Delta}$ appears as the ``cross term'' of this expression. The notation $B_{\Delta}$ comes from Gross-Zagier in \cite{GZ85}.
\end{remark}

\begin{remark}
The root geodesics intersect if and only if the cross-ratio $(q_{1,f},q_{1,s};q_{2,f},q_{2,s})$ is negative. A messy computation shows that this cross-ratio is equal to
\[\dfrac{B_{\Delta}(q_1,q_2)-\sqrt{D_1D_2}}{B_{\Delta}(q_1,q_2)+\sqrt{D_1D_2}},\]
which provides an alternative proof of Proposition \ref{propsl2zintersect}i.
\end{remark}

\subsection{Intersection numbers as linking numbers}\label{subseclink}
\begin{sloppypar}
The proof that intersection numbers of modular geodesics coincide with linking numbers in $\SL(2,\Z)\backslash\SL(2,\R)$ is originally due to Birkhoff in \cite{BK17}. Here, we follow \cite{DIT17}, and make this connection concrete.
\end{sloppypar}

Let $\gamma\in\SL(2,\Z)$ be a primitive hyperbolic matrix with positive trace, and take $M_{\gamma}$ to be a diagonalization matrix, so that
\[\gamma M_{\gamma}=M_{\gamma}\left(\begin{matrix}\epsilon & 0\\0 &\frac{1}{\epsilon}\end{matrix}\right),\]
where $\epsilon>1$ is the larger eigenvalue of $\gamma$. Let $\phi(t):=\sm{e^t}{0}{0}{e^{-t}}$, and then
\[\tilde{\gamma}_{+}(t):=M_{\gamma}\phi(t)\quad\text{and}\quad \tilde{\gamma}_{-}(t):=M_{\gamma}S\phi(t)\quad\text{for $0\leq t\leq\log(\epsilon)$}\]
define closed paths in the space $\SL(2,\Z)\backslash\SL(2,\R)$. The knot $[\tilde{\gamma}_{+}]$ is null-homologous if and only if $\gamma$ is reciprocal, i.e. it is conjugate to $\gamma^{-1}$. Thus we consider the link
\[[\tilde{\gamma}]:=[\tilde{\gamma}_{+}]+[\tilde{\gamma}_{-}],\]
which is null-homologous in $\SL(2,\Z)\backslash\SL(2,\R)$. Furthermore, this link remains constant over a $\SL(2,\Z)$ conjugacy class, as well as replacing $\gamma$ by $\gamma^{-1}$.

Given a pair of strongly inequivalent conjugacy classes of primitive hyperbolic matrices, say $\mathcal{C}_{\sigma},\mathcal{C}_{\gamma}$, their linking number is the linking number of the null-homologous links associated to the classes. Denote this by $\Lk(\mathcal{C}_{\sigma},\mathcal{C}_{\gamma})$.

\begin{theorem}
We have
\[\Lk(\mathcal{C}_{\sigma},\mathcal{C}_{\gamma})=-\Int_{\PSL(2,\Z)}^1(\sigma,\gamma),\]
and the linking number is always even. Furthermore, if $\sigma$ or $\gamma$ is reciprocal, then the linking number is a multiple of $4$.
\end{theorem}
\begin{proof}
Following \cite{DIT17}, let $\mathcal{C}$ be a conjugacy class of primitive hyperbolic matrices, and take any $\sigma\in\mathcal{C}$. Let $\Gamma=\SL (2,\Z)$ and let $\Gamma_{\sigma}=\{g\in\Gamma:g\sigma g^{-1}=\sigma\}=\pm\sigma^{\Z}$. For $z_1,z_2\in\overline{\mathbb{H}}$, define
\[I_{\mathcal{C}}(z_1,z_2):=\{\alpha\in\Gamma\backslash\Gamma_{\sigma}:\alpha \ell_{\sigma}\text{ intersects }\dot{\ell}_{z_1,z_2}\},\]
and note that its size does not depend on the choice of $\sigma\in\mathcal{C}$. Theorem $6.4$ of \cite{DIT17} shows that taking $z_0=M_{\gamma}i\in\ell_{\gamma}$ gives
\[\Lk(\mathcal{C}_{\sigma},\mathcal{C}_{\gamma})=-|I_{\mathcal{C}_{\sigma}}(z_0,\gamma z_0)|.\]
For $\alpha\in I_{\mathcal{C}}(z_0,\gamma z_0)$, the root geodesic $\alpha\ell_{\sigma}=\ell_{\alpha\sigma\alpha^{-1}}$ intersects $\dot{\ell}_{z_0,\gamma z_0}$, and $\alpha\sigma\alpha^{-1}$ is well defined and distinct $\alpha$'s give distinct conjugates (since $\alpha\in \Gamma\backslash\Gamma_{\sigma}$).

Proposition \ref{propfourthintersection} showed that
\[\Int_{\PSL(2,\Z)}^{1}(\sigma,\gamma)=\sum_{\substack{\beta\in\Gamma\text{ conjugate to }\sigma\\ |\ell_{\beta}\pitchfork\dot{\ell}_{z,\gamma z}|=1}}1,\]
for $z\in\ell_{\gamma}$ not a fixed point of $\gamma$. The first result follows by taking $z=z_0$. The rest will follow from Corollary \ref{corintiseven}.
\end{proof}

\subsection{Gross-Zagier analogue}

In \cite{GZ85}, Gross and Zagier define the function
\[p_{GZ}(n)=\frac{1}{2}\left|\left\{(q_1,q_2):\disc(q_1)=D_1, \disc(q_2)=D_2, B_{\Delta}(q_1,q_2)=-n\right\}\diagup\sim_2\right|,\]
where $D_1,D_2$ are negative discriminants, $n>\sqrt{D_1D_2}$, and $n\equiv D_1D_2\pmod{2}$. Proposition $6.1$ of \cite{GZ85} says that if $D_1,D_2$ are coprime and fundamental, then
\[p_{GZ}(n)=\sum_{d\mid\frac{n^2-D_1D_2}{4}}\epsilon(d).\]
Recall that for positive discriminants $D_1,D_2$ we have defined the function $p(n)$ as
\[p(n):=\left|\left\{(q_1,q_2):\disc(q_1)=D_1, \disc(q_2)=D_2, B_{\Delta}(q_1,q_2)=n\right\}\diagup\sim_2\right|,\]
and in Theorem \ref{mainthm1} claimed that
\[p(n)=2\sum_{d\mid\frac{D_1D_2-n^2}{4}}\epsilon(d),\]
when $D_1, D_2$ are also coprime and fundamental. In particular, it is clear that $p(n)$ is the real quadratic analogue of $p_{GZ}(n)$, and that Theorem \ref{mainthm1} is the real quadratic analogue of Proposition $6.1$ of \cite{GZ85}. While it may be possible to adopt their proof of Proposition $6.1$ to this situation, there are complications with the exact sequence involving class groups and unit groups found on page $213$ of \cite{GZ85}. The sequence and proof do not translate exactly to our case; for example, the unit group of $L$ now has rank $3$ instead of $1$.

As a result, we will sketch a different proof of the theorem (see Sections 4-6 of \cite{JR21shim} for full details).
\begin{proof}[Sketch of the proof of Theorem \ref{mainthm1}]
Let $q=[A,B,C]$ be a PIBQF of discriminant $D$, and let $\mathcal{O}_D$ be the unique quadratic order of discriminant $D$. We get a corresponding embedding $\phi_q:\mathcal{O}_D\rightarrow M_2(\Z)$ induced from
\[\phi_q(\sqrt{D})=\left(\begin{matrix} -B & -2C\\2A & B\end{matrix}\right).\]
Now,
\begin{enumerate}
\item $B_{\Delta}(q_1,q_2)=n$ is equivalent to $\frac{1}{2}\Tr\left(\phi_{q_1}(\sqrt{D_1})\phi_{q_2}(\sqrt{D_2})\right)=n$.
\item By completing at all places and using the local-global principle, it follows that $p(n)=0$ if and only if $r>0$ (coprimality is used for the only if, and fundamentalness guarantees the resulting embeddings correspond to primitive forms of discriminants $D_1,D_2$).
\item If $r=0$, fix a pair $(q_1,q_2)$ of PIBQFs of discriminants $D_1,D_2$ respectively with $B_{\Delta}(q_1,q_2)=n$.
\item If $(q_1',q_2')$ is another such pair, show that the pairs of embeddings $(\phi_{q_1},\phi_{q_2})$ and $(\phi_{q_1'},\phi_{q_2'})$ are related by a simultaneous conjugation in $\SL(2,\Q)$.
\item Let $\mathbb{O}$ be the smallest order of $M_2(\Q )$ which contains $\phi_{q_1}(\mathcal{O}_{D_1})$ and $\phi_{q_2}(\mathcal{O}_{D_2})$.
\item By explicitly demonstrating a basis for $\mathbb{O}$, calculate that its discriminant is $\frac{D_1D_2-n^2}{4}$.
\item For every maximal order containing $\mathbb{O}$, we can conjugate it to make $M_2(\Z)$. Since the stabilizer (under conjugation) of $M_2(\Z)$ in $M_2(\Q)$ is $\GL(2,\Z)$, we obtain two $\SL(2,\Z)$ equivalence classes of pairs of embeddings.
\item Each equivalence class of pairs of embeddings corresponds to a unique pair of PIBQFs counted in $p(n)$. The fourth point implies that all pairs in $p(n)$ are counted, hence $\frac{p(n)}{2}$ bijects with the set of maximal orders containing $\mathbb{O}$.
\item The number of maximal orders containing $\mathbb{O}$ can be calculated locally. At the primes $q_i$, the completed order $\mathbb{O}_{q_i}$ is not contained in any Eichler order of $M_2(\Q _{q_i})$, whence it is contained in exactly one maximal order.
\item The order $\mathbb{O}_{w_i}$ is an Eichler order of level $w_i^{g_i}$ in $M_2(\Q_{w_i})$, and is thus contained in $g_i+1$ maximal orders.
\item Thus $p(n)=2\prod_{i=1}^t (g_i+1)$ if $r=0$ and $p(n)=0$ if $r>0$.
\item Let $f(k)=\sum_{d\mid k}\epsilon(d)$ for all $k$ for which $\epsilon(k)$ is defined. Note that $f$ is multiplicative, and that
\[f(p_i^{2e_i+1})=0,\,\,\,\,f(q_i^{2f_i})=1,\,\,\,\,f(w_i^{g_i})=g_i+1,\]
and hence $p(n)=2f\left(\frac{D_1D_2-n^2}{4}\right)$, as desired.
\end{enumerate}

The equality $f\left(\frac{D_1D_2-n^2}{4}\right)=r_{L/K}\left(\frac{\sqrt{D_1D_2}-n}{2}\right)$ follows either from similar arguments to Gross-Zagier, or combining point 11 with a case analysis of how primes dividing $\langle\frac{\sqrt{D_1D_2}-n}{2}\rangle$ split in $L$.
\end{proof}

In Corollary \ref{corintersectionbounds}, we will prove that the intersection number is always non-zero. Combining this with the above theorem gives an interesting little result.

\begin{corollary}
Let $D_1,D_2$ be positive coprime fundamental discriminants. Then there exists a non-negative integer $n$ such that
\begin{itemize}
\item $n<\sqrt{D_1D_2}$ and $n\equiv D_1D_2\pmod{2}$;
\item If $p$ is a prime for which $\epsilon(p)=-1$, then $v_p\left(\frac{D_1D_2-n^2}{4}\right)$ is even.
\end{itemize}
\end{corollary}

\section{Computing intersection numbers in terms of the Conway topograph}\label{secconway}
The Conway topograph is a device used to understand the equivalence class of a binary quadratic form. For an alternate presentation of the Conway topograph (as well as another interpretation of the river of an indefinite form), see \cite{SV18}.

\subsection{The action of PSL(2,Z) on an infinite 3-regular tree}
Let $G$ be the infinite $3-$regular connected tree drawn in the plane. The Conway topograph will consist of $G$ and some additional data; we first study $G$. Let $E(G)_{\text{or}}$ be the set of pairs $(E,V)$ where $E$ is an edge of $G$, and $V$ is one of the two vertices on $E$ (i.e. an oriented edge). We will define an action of $\PSL(2,\Z)$ on $E(G)_{\text{or}}$.

Recall that $\PSL(2,\Z)$ is generated by the matrices $S=\sm{0}{1}{-1}{0}$ and $T=\sm{1}{1}{0}{1}$, which gives the group presentation
\[\PSL (2,\Z)=\langle S,T| S^2=(ST)^3=1\rangle.\]
If $E$ has vertices $V_1,V_2$, define the action of $S$ on $(E,V_1)$ to be $S\circ(E,V_1)=(E,V_2)$, or equivalently you swap the edge orientation. To act via $T$, move along $E$ to $V_1$, and take the left branch with the same orientation (i.e. the vertex that is not $V_1$). With reference to Figure \ref{figactionofT}, the action of $T$ is
\[T\circ (E,V_1)=(F,V_3),\qquad T\circ (F,V_1)=(G,V_4),\qquad T\circ (G,V_1)=(E,V_2).\]

\begin{figure}[ht]
	\centering
		\includegraphics[scale=0.6]{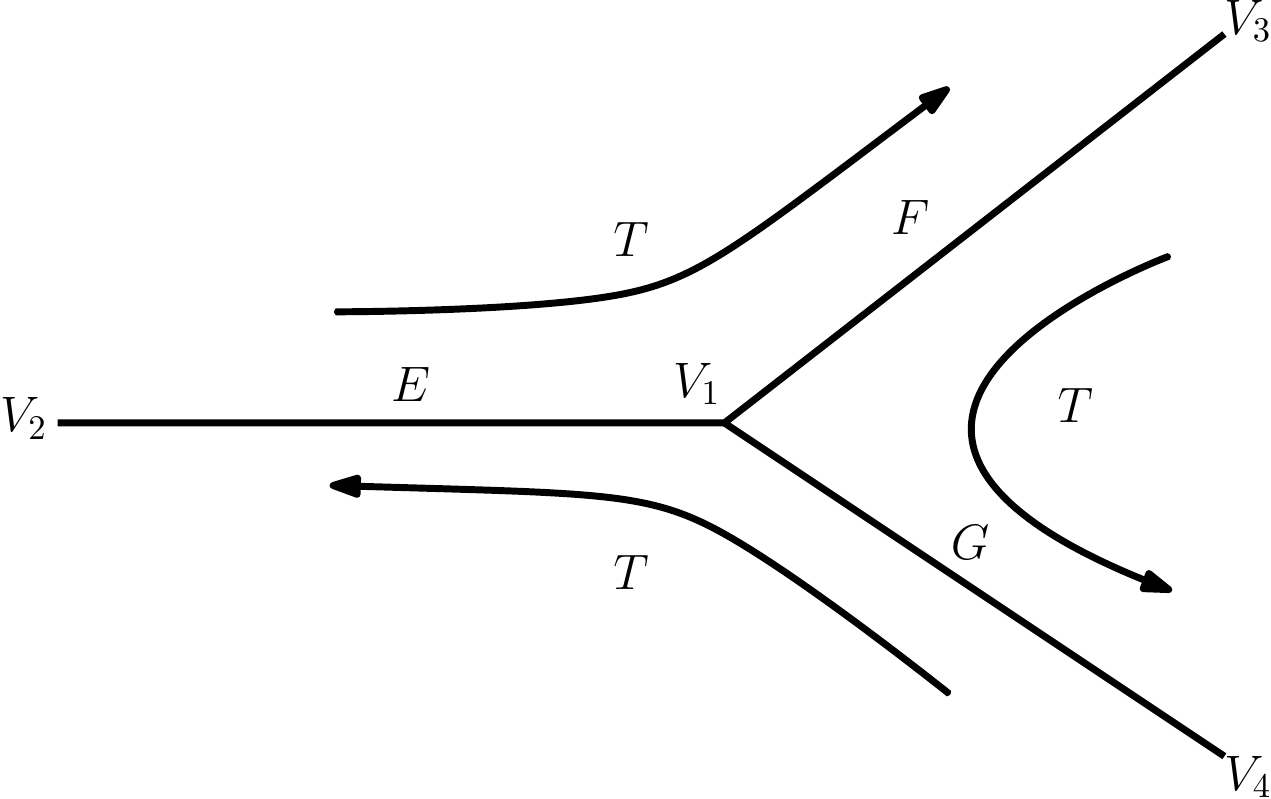}
	\caption{Action of $T$.}\label{figactionofT}
\end{figure}

From this it is easy to see that the relations $S^2=(ST)^3=1$ are satisfied, i.e. our action does descend down to an action of $\PSL(2,\Z)$. Furthermore, it is clear that the action is transitive, and the stabilizer of $(E,V)$ is trivial. Thus we can form a (non-canonical) bijection between $E(G)_{\text{or}}$ and $\PSL (2,\Z)$, by picking a base element of $E(G)_{\text{or}}$.

Note that an alternate interpretation of $E(G)_{\text{or}}$ is as the set of ordered triples $(R_1,E,R_2)$, where $R_1,R_2$ are distinct regions of the plane formed by $G$ that are separated by the edge $E$.

\subsection{Definition of the topograph}
A completed topograph will consist of the graph $G$, with numbers in all the regions formed by $G$, numbers on all of the edges, and arrows on certain edges. To read off a BQF, pick any region $R_1$ and edge $E$ bordering the region, and let $R_2$ be the region on the other side of $E$. Orient so that $E$ is horizontal, with $R_1$ above $E$ and $R_2$ below it. If $r_i$ and $e$ represent the numbers on the regions and edge, then we form the BQF $[r_1,e,r_2]$ if the arrow on $E$ is pointing right, and $[r_1,-e,r_2]$ if the arrow is pointing left. There will be no arrow if and only if $e=0$, and then you form $[r_1,0,r_2]$. The BQFs read off in this fashion will form an entire equivalence class of BQFs. Figure \ref{figreadBQF} is an example of how to read off BQFs.
\begin{figure}[ht]
	\centering
		\includegraphics[scale=1]{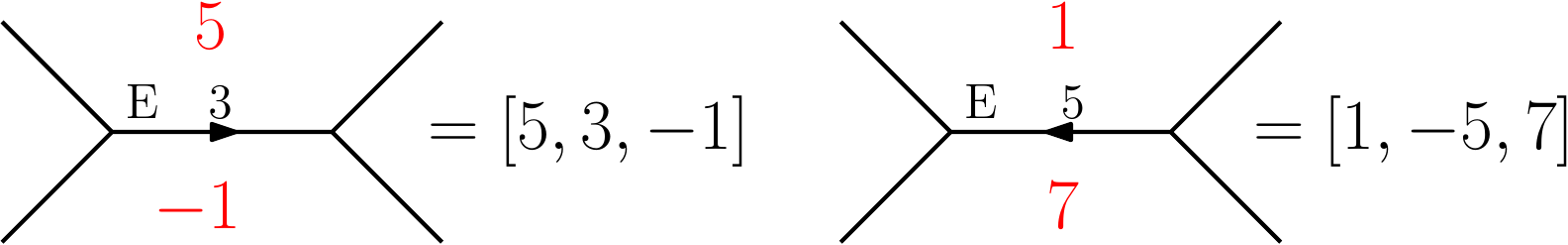}
	\caption{Reading BQFs from the topograph; region $R_1$ is above $E$, and region $R_2$ is below $E$.}\label{figreadBQF}
\end{figure}

To create the topograph, start with a BQF $q=[A,B,C]$, and pick any pair $(E,V)\in E(G)_{\text{or}}$. For $M\in\PSL (2,\Z)$, let $M\circ (E,V)=(E',V')$. When $E'$ is horizontal with $V'$ on the right, let $R_1$ be the region above $E'$ and $R_2$ be the region below $E'$. If $M\circ [A,B,C]=[A',B',C']$, we write the number $A'$ in $R_1$, and $|B'|$ on $E'$. If $B'>0$, draw the arrow so that when $E'$ is horizontal with $R_1$ above it, then the arrow points right. If $B'<0$ draw the opposite arrow, and if $B'=0$ draw no arrow.

First, we claim that this is well defined. Consider the equations
\begin{equation}\label{eqnactionofTS}
T\circ[A,B,C]=[A,B+2A,A+B+C],\qquad S\circ[A,B,C]=[C,-B,A].
\end{equation}
If $M,M'$ correspond to the same region $R_1$, then we necessarily have $M'=MT^k$ for some integer $k$, and Equation \ref{eqnactionofTS} implies that they define the same number. If $M,M'$ correspond to the same edge, we either have $M'=M$ or $M'=MS$, and Equation \ref{eqnactionofTS} again implies that the definition of $|B|$ and the arrow was consistent.

For an alternate interpretation of the arrow, note that each edge $E$ touches four regions, two along the length of the edge, and two its vertices. The arrow on $E$ points from the region touching a vertex with a smaller number to the region touching a vertex with a larger number. These regions have the same number if and only if the number on the edge is $0$, i.e. no arrow was drawn.

As examples, Figures \ref{fig102top} and \ref{fig12m2top} are parts of the topographs for the forms $[1,0,-2]$ and $[1,2,-2]$. The numbers in the regions are coloured red, and the numbers on the edges are black.

\begin{figure}[!ht]
\centering
\begin{minipage}{.5\textwidth}
	\centering
  \includegraphics[width=.8\linewidth]{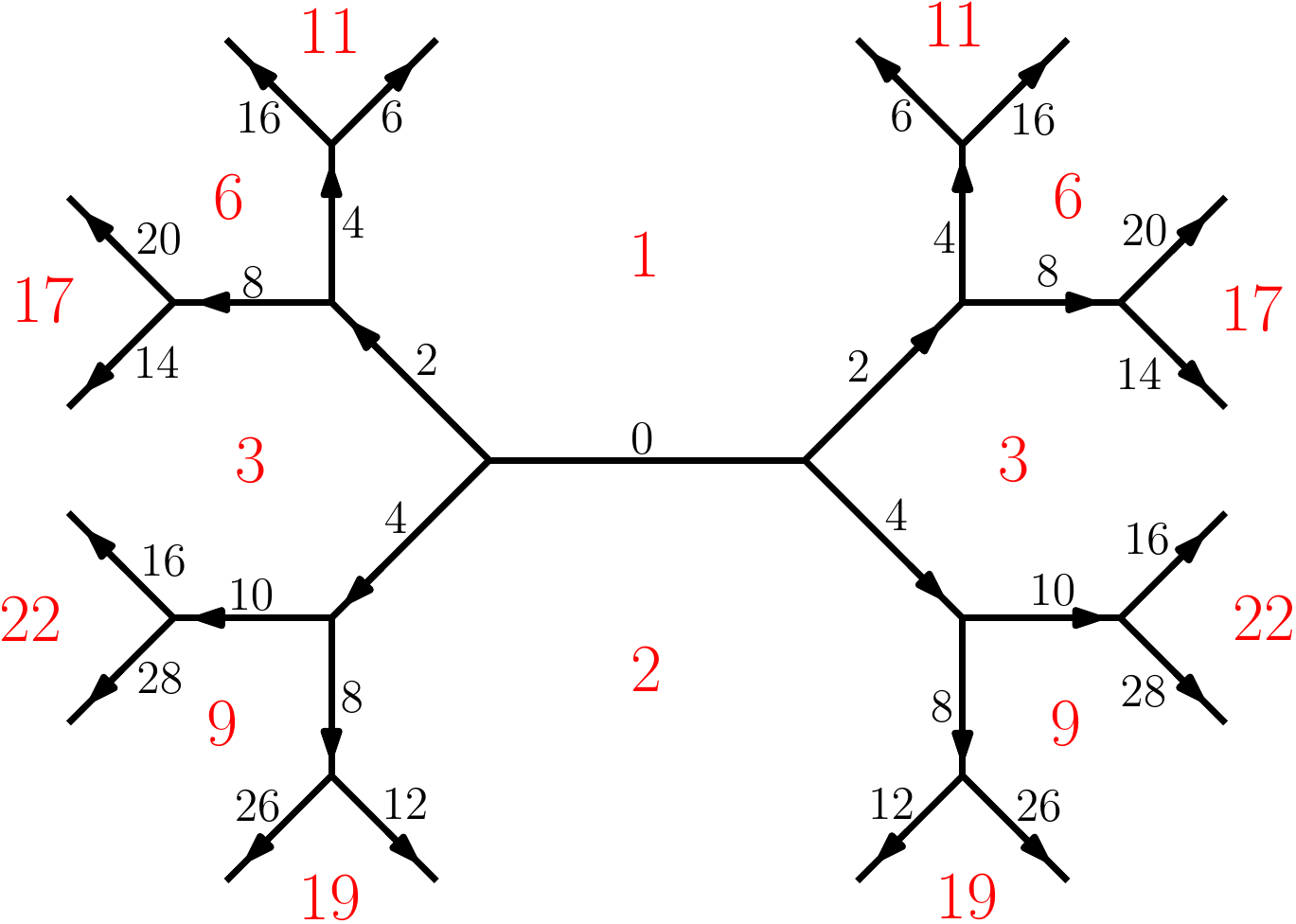}
  \captionof{figure}{$[1,0,2]$ topograph.}\label{fig102top}
\end{minipage}%
\begin{minipage}{.5\textwidth}
  \centering
  \includegraphics[width=\linewidth]{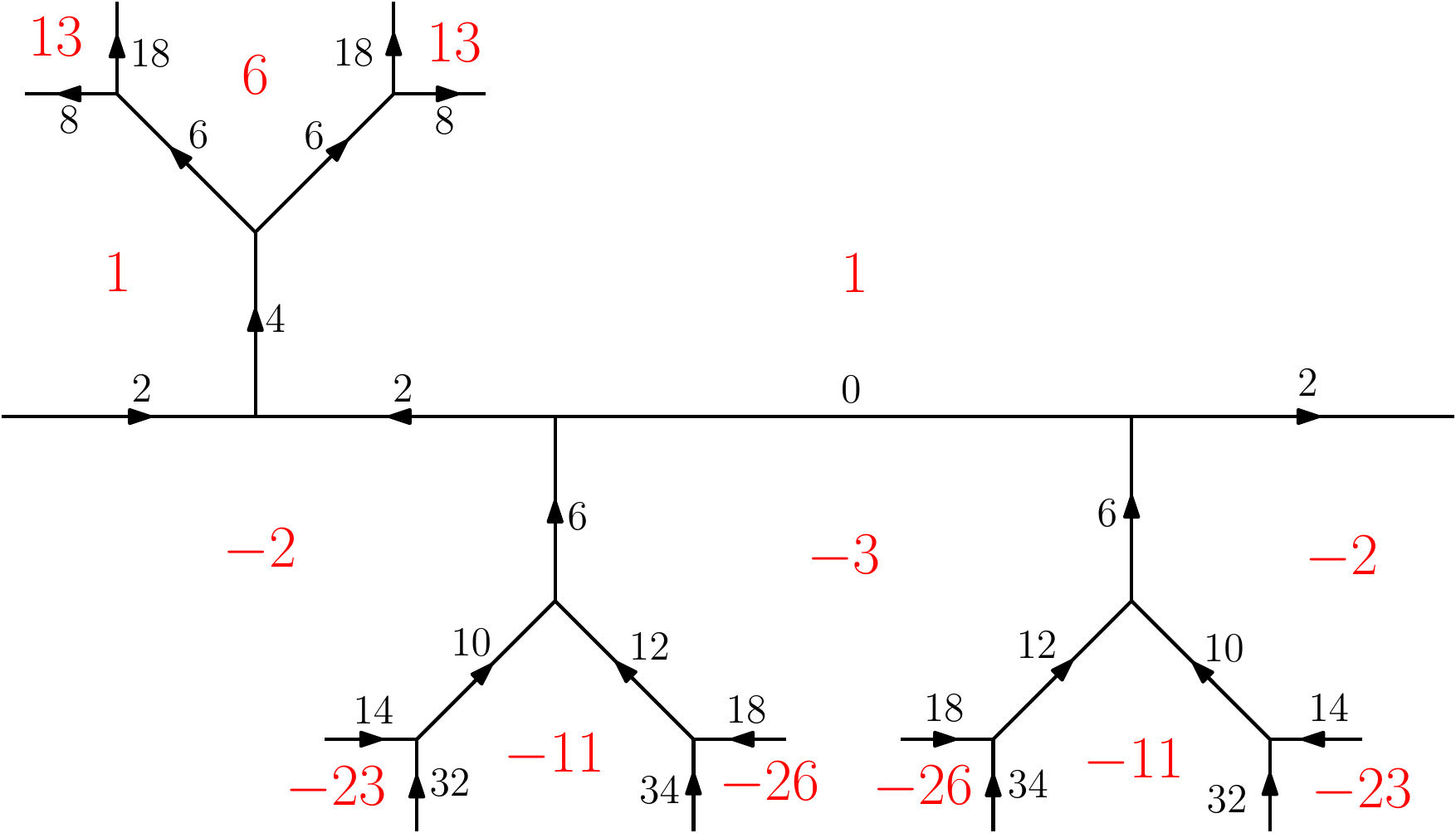}
  \captionof{figure}{$[1,2,-2]$ topograph.}\label{fig12m2top}
\end{minipage}
\end{figure}

\subsection{Key properties of the topograph}

From the construction, it is immediate that the BQFs read off from the topograph created from $[A,B,C]$ form the equivalence class for forms similar to $[A,B,C]$. Furthermore, any two forms in this equivalence class produce isomorphic topographs. However, one must be careful, as a form $[A,B,C]$ does not necessarily correspond to a \textit{unique} pair $(E,V)$ on its topograph. Indeed, it appears uniquely if and only if $[A,B,C]$ has trivial automorph. Assuming the form is primitive, this happens if and only if $D<-4$.

The numbers which appear in regions are precisely the numbers which can be represented properly by the BQFs in the equivalence class. This fact, coupled with the following lemma, allows us to determine if a number is properly represented by a given BQF.

\begin{lemma}[Climbing Lemma]
Let $q=[A,B,C]$ be a BQF with $A,B,C>0$. In the topgraph with $q$ present, numbers beyond $q$ (in the direction of the arrow on the edge corresponding to $q$) are strictly increasing.
\end{lemma}
\begin{proof}
In the region beyond we fill in $A+B+C$, and on the two adjacent edges we fill in $B+2A,B+2C$, so the numbers in the regions and the edges grow (and they remain positive, so the same applies again).
\end{proof}

Given two forms on a topograph $q_1,q_2$, we can easily find the transition matrix to go between them. Indeed, let $M=\text{Id}$, let
\begin{equation}\label{eqnRLdef}
L=T=\left(\begin{matrix}1&1\\0&1\end{matrix}\right),\qquad R=\left(\begin{matrix}1&0\\1&1\end{matrix}\right),
\end{equation}
and start at the oriented edge corresponding to $q_1$. Take the path to the oriented edge corresponding to $q_2$, where going forward and left corresponds to multiplying $M$ by $L$ on the right, forward and right corresponds to multiplying $M$ by $R$ on the right, and reversing direction corresponds to multiplying $M$ by $S$ on the right. 

From the $[1,0,2]$ topograph found in Figure \ref{fig102top}, take $q_1=[2,4,3]$ and $q_2=[17,14,3]$, and we find that
\[M=SRLRR=\left(\begin{matrix}5&2\\-3&-1\end{matrix}\right).\]
It can be checked that indeed, $M\circ q_1=q_2$.

\subsection{The topograph of indefinite forms}
When a binary quadratic form is indefinite, it will properly represent both positive and negative numbers. How is this fact reflected in the topograph? First, note that there are finitely many forms $[A,B,C]$ of fixed discriminant $D>0$ which satisfy $AC<0$, since the equation $D=B^2-4AC> B^2$ must be satisfied. On the topograph, it can be shown that such forms form a single path called the ``river,'' which separates the regions with positive numbers from the regions with negative numbers. Since there are finitely many forms possible, it is in fact a periodic sequence.

When drawing the topograph of an indefinite form, it is best to ``flatten'' the river and draw it horizontally, with trees branching off above (the positive direction) and below (the negative direction). Start at a vertex $V$ on the river, and travel to the right along it, keeping track at each vertex whether we go left ($L$) or right ($R$). By stopping after the going along the smallest period of the river, we get a sequence of $L$'s and $R$'s. By taking a sequence to be equivalent to cyclic shifts, we assign a sequence to each topograph. 

\begin{definition}
For a topograph $T$ or form $f$ in $T$, define $\Riv (T)=\Riv (f)$ to be this sequence, called the ``river sequence''. It can either be thought of as an infinite (in both directions) periodic sequence of $L$'s and $R$'s, or as a finite sequence by only taking the least period of the topograph river, and declaring two sequences to be the same if they differ by a cyclic shift (for example $LLR\sim LRL\sim RLL$).
Furthermore, the river is said to ``flow'' from left to right when the positive regions are above the river.
\end{definition}

The topograph of $[1,2,-2]$ as displayed in Figure \ref{fig12m2top} has river sequence $RLL$, and Figure \ref{fig1014m5top} gives part of the topograph of $[10,14,-5]$, which has discriminant $396$ and river sequence $RRRLLRL$.
\begin{figure}[ht]
	\centering
		\includegraphics[scale=0.9]{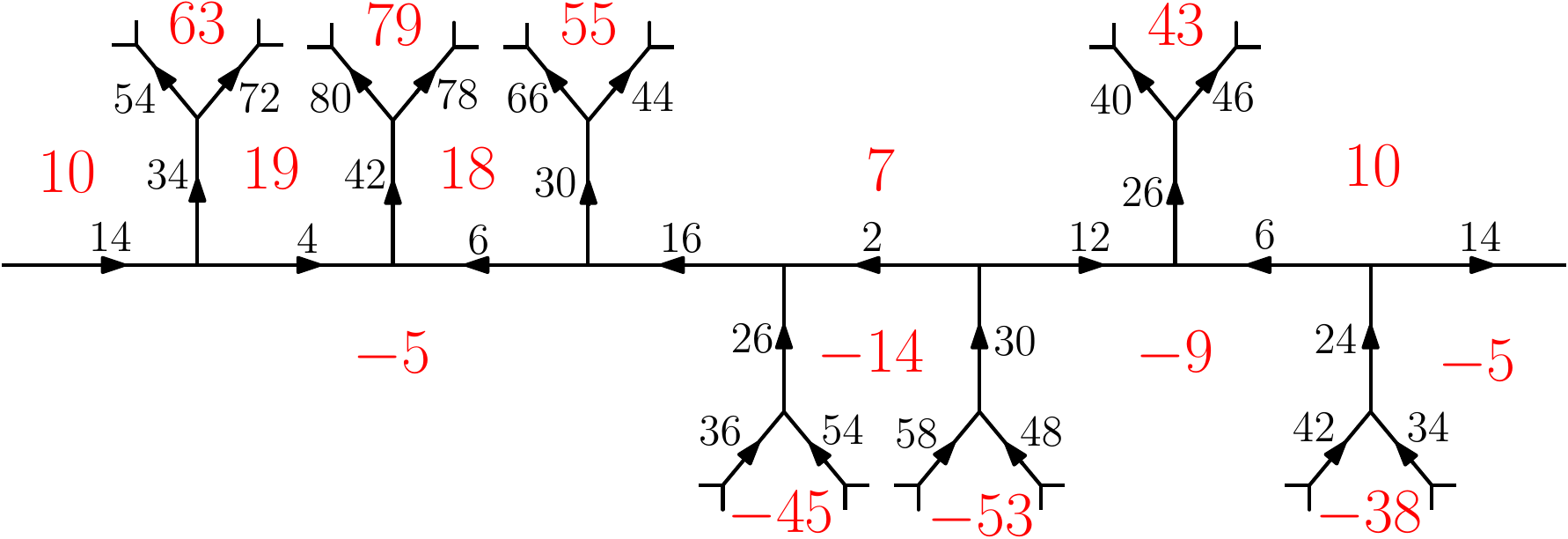}
	\caption{$[10,14,-5]$ topograph.}\label{fig1014m5top}
\end{figure}

Some key questions are:
\begin{itemize}
\item Can we recover a topograph from a river sequence?
\item What is the connection to a topograph where the river ``flows backwards''?
\item What river sequences are possible?
\item Does $\gamma_q$ flow with or against the river?
\end{itemize}

The answer to the first question is yes. Take the smallest period of the river sequence, and using $R,L$ as in Equation \ref{eqnRLdef}, we get the invariant automorph of a form on the river, which thus determines the form and hence the entire topograph. It is important that we constructed this automorph by going \textit{right} in the sequence, i.e. in the direction of the flow of the river. If we had gone to the left (against the flow), we would have also picked up a generator of the automorphism group of $q$, but it would be the inverse of what we define in Definition \ref{defautom}. In fact, this shows that this river sequence (i.e. going left) gives us the river sequence of $-q$, the reciprocal of $q$. Formally put, to get $\Riv (-q)$, take $\Riv (q)$, replace the $L$'s by $R$'s and $R$'s by $L$'s, and reverse the sequence. Thus it is easy to detect if an equivalence class is reciprocal directly from the river sequence. For example, $RRRL$ is not reciprocal, but $RRLL$ is.

In terms of possible sequences, note that there must be at least one $L$ and one $R$, as indefinite forms represent both positive and negative numbers. From the above commentary, we see that any periodic sequence with at least one $L$ and one $R$ is the river of some topograph (noting that the constructed automorph is in fact hyperbolic, so it does correspond to a PIBQF).

The final answer is $\gamma_q$ flows with the river, no matter what $q$ is. To see this, first assume $q=[A,B,C]$ is on the river with $A>0$. Then the entries of $\gamma_q$ are all positive, and two of the enties of $\gamma_q^{-1}$ are negative. The invariant automorph obtained by going along the flow of the river will be a product of $L$'s and $R$'s, and will thus have positive entries, which gives the result in this case. The general result follows from Equation \ref{eqnbqfmatequivcorr}.

\begin{remark}
When studying indefinite quadratic forms, one normally introduces the notion of a reduced form, defines right and left neighbours of reduced forms, and shows that this forms a unique cycle. Taking the common choice of $[A,B,C]$ is reduced if $AC<0$ and $B>|A+C|$, when going along the river, these reduced forms correspond to the forms between the branches switching from the negative to the positive sides of the river (and vice versa). Taking the right/left neighbour just corresponds to going to the next reduced form along the river.
\end{remark}

\begin{remark}\label{remriverroots}
Let $q=[A,B,C]$ be a PIBQF; we can think of the river of the topograph of $q$ as its root geodesic. We have $A=q(1,0)$, and the number appearing in the corresponding place after applying $\gamma_q^n$ will be $A=q(x,y)$, where 
\[\left(\begin{smallmatrix}x\\y\end{smallmatrix}\right)=\gamma_q^n\left(\begin{smallmatrix}1\\0\end{smallmatrix}\right),\]
whence $\frac{x}{y}=\gamma_q^n(\infty)$.
As $n\rightarrow\infty$, $\gamma_q^n(\infty)\rightarrow q_f$, and as $n\rightarrow -\infty$, $\gamma_q^n(\infty)\rightarrow q_s$. Since $\gamma_q$ moves along the river in the direction it is flowing, we can think of the river as flowing from the second root of $q$ to the first root of $q$.
\end{remark}

\begin{remark}\label{remcontinuedfrac}
The river sequence of the topograph is directly connected to continued fractions. Let the continued fraction of $q_f$ be 
\[[a_0,a_1,\ldots]=[a_0,a_1,\ldots,a_s,\overline{a_{s+1},\ldots,a_{s+p}}],\]
where $s$ is the smallest integer such that the continued fraction is periodic after index $s$, and $p$ is the smallest \textit{even} integer such that the sequence has period $p$. Define $\Riv'(q)$ to be the sequence of $0$'s and $1$'s formed by:
\begin{itemize}
\item $s+1\pmod{2}$ repeated $a_{s+1}$ times;
\item $s+2\pmod{2}$ repeated $a_{s+2}$ times;
\item $\cdots$
\item $s+p\pmod{2}$ repeated $a_{s+p}$ times.
\end{itemize} 
Then $\Riv'(q)$ and $\Riv(q)$ are equal, when we identify $1$ with $R$ and $0$ with $L$. For example, if $q=[10,14,-5]$, then the continued fraction of $q_f$ is 
\[[0,\overline{3,2,1,1}],\]
 and 
\[\Riv'(q)=(1,1,1,0,0,1,0),\]
which agrees with Figure \ref{fig1014m5top}.
\end{remark}

\subsection{Intersection numbers in terms of the topograph}\label{sectopoint}
Consider two topographs $T_i$, with chosen pairs $(E_i,V_i)$ of an edge $E_i$ in the graph of $T_i$ and a vertex $V_i$ on $E_i$ ($i=1,2$). Since the underlying graphs are the same, we can superimpose one graph on the other by identifying $V_1$ with $V_2$ and $E_1$ with $E_2$. When we do this, one can consider the interaction of the superimposed rivers $R_1,R_2$.

\begin{proposition}\label{propintriver}
Let the PIBQFs $q_i$ correspond to topographs $T_i$ ($i=1,2$). 
\begin{enumerate}[label=(\roman*)]
\item The root geodesics of $q_1,q_2$ intersect uniquely in the upper half plane if and only if when you superimpose $T_1,T_2$ at $q_1,q_2$ (as above), the superimposed rivers $R_1,R_2$ meet and cross. Furthermore, the root geodesics completely overlap if and only if the rivers $R_1,R_2$ completely overlap.
\item If the root geodesics of $q_1,q_2$ intersect uniquely in the upper half plane, consider the flow of the rivers. Going along the river $R_1$ in the direction it is flowing, if $R_2$ joins the river from the right hand side then the sign of the intersection is $1$, and if it joins from the left the sign is $-1$.
\end{enumerate}
\end{proposition}
\begin{proof}
Consider the set
\[\{(\sign (q_1(x,1)),\sign (q_2(x,1)))\},\]
as $x$ ranges over $\R$. Of the $4$ possible non-zero pairs of signs $(\pm 1,\pm 1)$, we note that
\begin{itemize}
\item All $4$ pairs appear if the root geodesics intersect in the upper half plane;
\item $3$ pairs appear if the root geodesics do not intersect in the upper half plane and do not overlap;
\item $2$ pairs appear if the root geodesics overlap.
\end{itemize}
This also remains true if we instead consider
\[\{(\sign (q_1(x,y)),\sign (q_2(x,y)))\},\]
where $(x,y)$ range over pairs of coprime integers.

When we superimpose the topographs, the numbers in the regions correspond to the values that $q_1,q_2$ take on coprime integers. Since we impose $q_1$ on top of $q_2$, the value of $q_1(x,y)$ is imposed onto the value of $q_2(x,y)$. However, the rivers $R_1,R_2$ determine the boundary between the signs of the numbers in the regions, so that we get $4$ sign combinations if and only if the rivers meet and cross, $3$ if they either never meet or meet and do not cross, and $2$ if they overlap. This completes the first half of the proposition.

\begin{figure}[ht]
	\centering
		\includegraphics[scale=0.7]{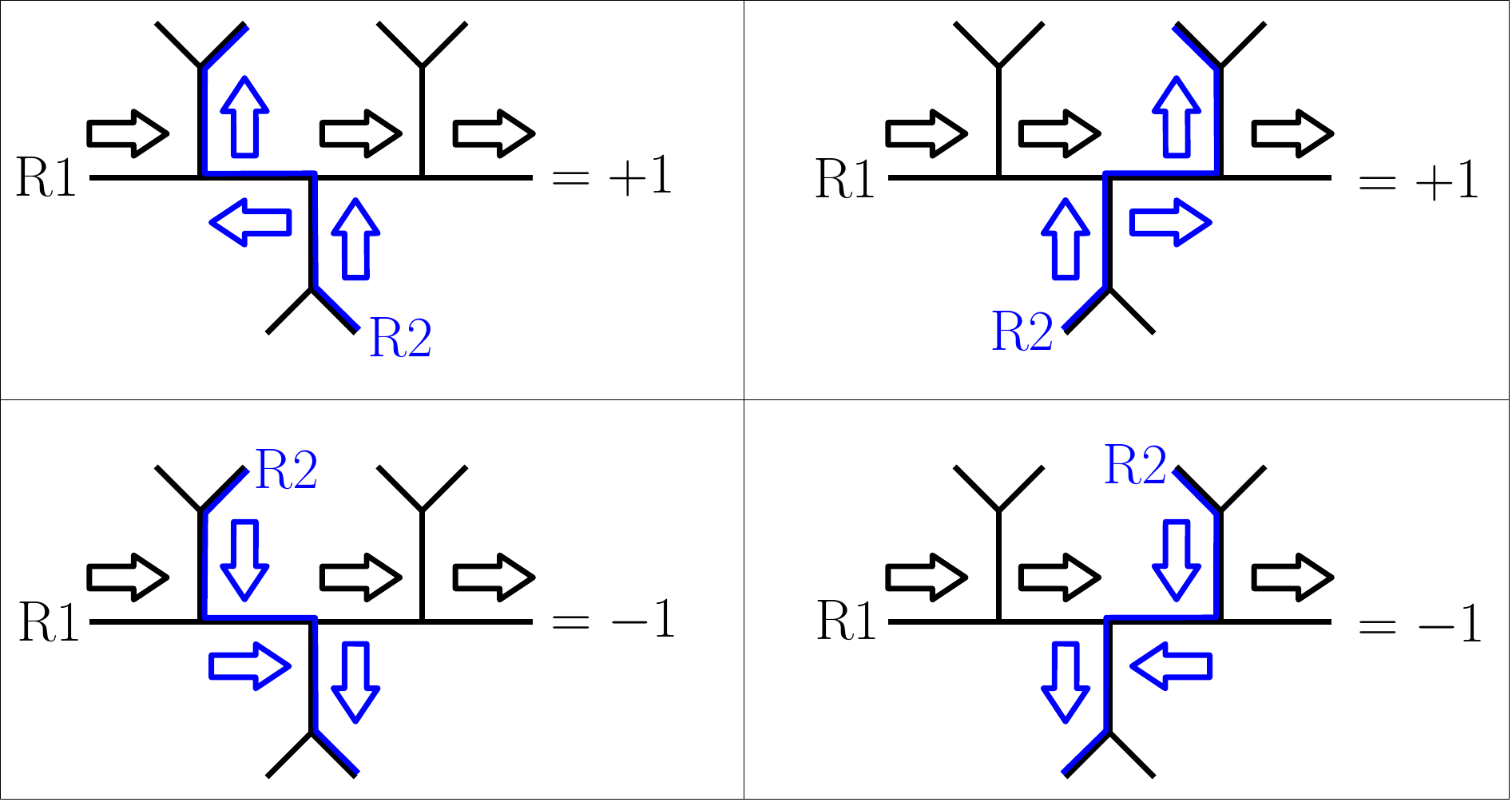}
	\caption{Possible flow configurations.}\label{figflowconfing}
\end{figure}

Figure \ref{figflowconfing} demonstrates the four possible flow configurations, and the claimed corresponding sign. Recall Remark \ref{remriverroots}, where we interpreted the river as flowing between the two roots. Combining this with the interpretation of the sign found in Remark \ref{remalternatesign} gives the above picture, as the second river originates at the second root and flows towards the first root.
\end{proof}

\begin{corollary}\label{corintriver}
Let $q_1,q_2$ be a pair of PIBQFs. Then the unweighted intersection number $\Int(q_1,q_2)$ is equal to the number of ways to superimpose the topographs corresponding to $q_1,q_2$ on top of each other so that the rivers $R_1,R_2$ meet and cross, modulo the periods of the rivers. The weighted intersection number $\Int^{\pm}(q_1,q_2)$ is the same, except we add $1$ when $R_2$ joins $R_1$ from the right, and $-1$ when $R_2$ joins $R_1$ from the left.
\end{corollary}
\begin{proof}
This follows from Proposition \ref{propdoublecoset} and Proposition \ref{propintriver}.
\end{proof}

Let's examine the consequences of Corollary \ref{corintriver} a bit more closely. If we have an intersection, we can follow the flow of the river $R_2$ until it meets $R_1$ to find a unique pair of vertices $(V_1,V_2)$ satisfying
\begin{itemize}
\item $V_i$ is on $R_i$ for $i=1,2$;
\item $V_1$ is superimposed on $V_2$;
\item the vertex preceding $V_2$ (in the sense of the flow of $R_2$) is not superimposed on the river $R_1$.
\end{itemize}
Furthermore, given a pair $(V_1,V_2)$ of vertices on the rivers $R_1,R_2$ respectively, there is a unique way to superimpose the topographs so the above is satisfied (though there is no guarantee that the rivers end up crossing). Since the rivers can either be flowing right or left at $V_i$, we have four different behaviours, and display them in Figure \ref{figriverintsign}.

\begin{figure}[ht]
	\centering
		\includegraphics[scale=0.8]{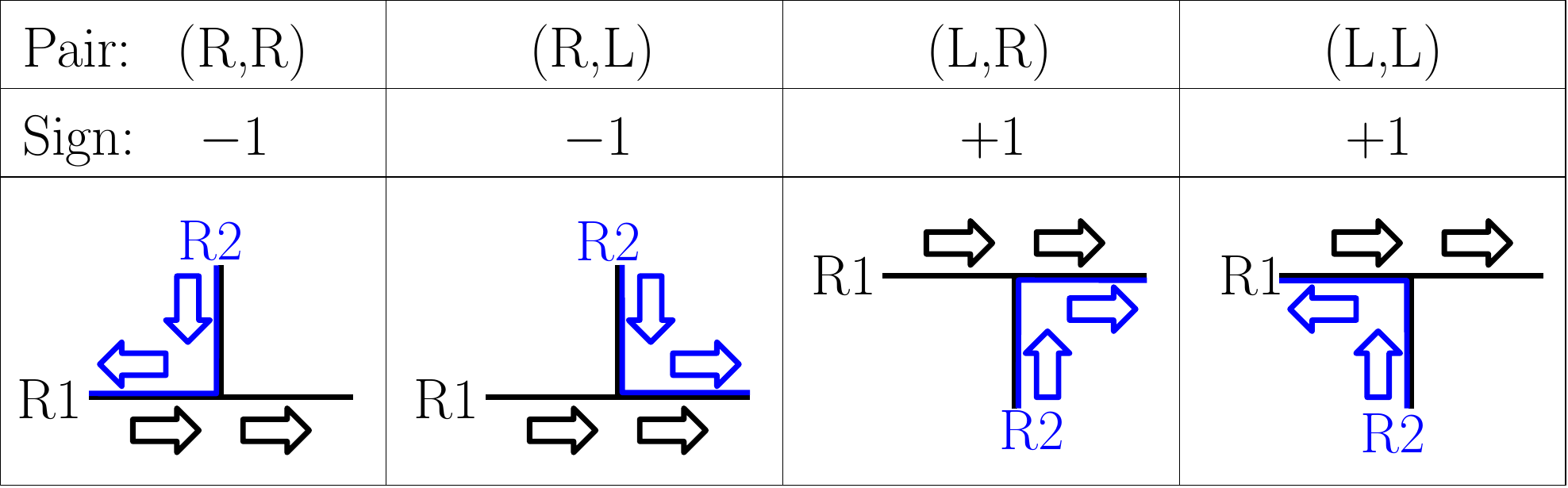}
	\caption{Sign of the intersection.}\label{figriverintsign}
\end{figure}

Put another way, the second river can join from the left (L) or right (R), and flow in the same (S) or opposite (O) direction over the course of their intersection. Figure \ref{figriverintsign} demonstrates the behaviours $LO,LS,RS,RO$ in order from left to right. For each $x\in\{LO,LS,RS,RO\}$, define $\Int^x$ to be the intersection number where we only count intersections of type $x$. In particular, we get that
\begin{equation}\label{eqnintsum}
\Int=\Int^{\text{RS}}+\Int^{\text{RO}}+\Int^{\text{LS}}+\Int^{\text{LO}}\text{ and }\Int^{\pm}=\Int^{\text{RS}}+\Int^{\text{RO}}-\Int^{\text{LS}}-\Int^{\text{LO}}.\end{equation}
Since the river of a reciprocal form is the same except with opposite flow, we can deduce that
\begin{equation}\label{eqnnegg}
\Int^{\text{RS}}(q_1,q_2)=\Int^{\text{LO}}(q_1,-q_2)\text{, and }\Int^{\text{RO}}(q_1,q_2)=\Int^{\text{LS}}(q_1,-q_2).\end{equation}
When we switch the order of $q_1,q_2$, we get
\[\Int^{\text{RS}}(q_1,q_2)=\Int^{\text{LS}}(q_2,q_1)\text{, and }\Int^{\text{RO}}(q_1,q_2)=\Int^{\text{LO}}(q_2,q_1).\]
For a non-trivial identity, we have the following proposition.

\begin{proposition}\label{propintsign0}
The following equalities hold
\[\Int^{\text{RS}}(q_1,q_2)=\Int^{\text{LS}}(q_1,q_2)\text{, and }\Int^{\text{RO}}(q_1,q_2)=\Int^{\text{LO}}(q_1,q_2).\]
\end{proposition}

\begin{proof}
Let the river corresponding to $q_1$ be $R_1=(x_1,x_2,\ldots,x_m)$ and the river corresponding to $q_2$ be $R_2=(y_1,y_2,\ldots,y_n)$, where $L$ is represented by $0$ and $R$ by $1$. For now, assume that $\gcd(m,n)=1$. Let $A=(a_1,a_2,\ldots,a_{mn})$ be the sequence $R_1$ repeated $n$ times, and let $B=(b_1,b_2,\ldots,b_{mn})$ be the sequence $R_2$ repeated $m$ times (take indices of $A,B$ modulo $mn$). As $\gcd(m,n)=1$, pairs $(x_i,y_j)$ with $1\leq i\leq m$ and $1\leq j\leq n$ biject with the pairs $(a_k,b_k)$ for $1\leq k\leq mn$. Intersecting the rivers flowing in the same direction at $(a_k,b_k)$ first requires $a_k\neq b_k$. The rivers will then take the same path until we get to the smallest $r\geq 1$ such that $a_{k+r}\neq b_{k+r}$. We will have an intersection if $a_k\neq a_{k+r}$!

In particular, consider the sequence $C=A+B\pmod{2}=(c_1,c_2,\ldots,c_{mn})$. Potential intersections will correspond to consecutive pairs of $1$'s in C, so let $I=\{1\leq i\leq mn:c_i=1\}=\{i_1,\ldots,i_r\}$ with $i_1<i_2<\ldots<i_r$. Form the sequence $D=(a_{i_1},a_{i_2},\ldots,a_{i_r})$, seen cyclically. In the sequence $D$, we have
\begin{itemize}
\item going from $0$ to $0$ corresponds to $R_2$ coming in from the right and leaving to the right;
\item going from $0$ to $1$ corresponds to $R_2$ coming in from the right and leaving to the left;
\item going from $1$ to $0$ corresponds to $R_2$ coming in from the left and leaving to the right;
\item going from $1$ to $1$ corresponds to $R_2$ coming in from the left and leaving to the left.
\end{itemize}
In particular, $\Int^{\text{RS}}(q_1,q_2)$ counts the number of times we change from $0$ to $1$ in $D$, and $\Int^{\text{LS}}(q_1,q_2)$ counts how many times we change from $1$ to $0$ in $D$. As $D$ is periodic, these are equal, hence the result follows in this case.

When $\gcd(m,n)=d>1$, we instead form sequences of length $\frac{mn}{d}=\lcm(m,n)$, and apply the above. Repeat with shifting the $B$ sequence by $1,2,\ldots,d-1$ to the right, and this covers all intersections.

The second statement follows from replacing $q_2$ by $-q_2$ and using Equation \ref{eqnnegg}.
\end{proof}

\begin{corollary}\label{corintiseven}
We have
\[\Int(q_1,q_2)=2(\Int^{\text{RS}}(q_1,q_2)+\Int^{\text{RO}}(q_1,q_2))\text{, and } \Int^{\pm}(q_1,q_2)=0.\]
Furthermore, if either $q_1$ or $q_2$ is reciprocal, then
\[\Int(q_1,q_2)=4\Int^{\text{RS}}(q_1,q_2).\]
\end{corollary}
\begin{proof}
This follows immediately from Equations \ref{eqnintsum},\ref{eqnnegg}, and Proposition \ref{propintsign0}. Note that $\Int^{\pm}(q_1,q_2)=0$ also follows from $\PSL(2,\Z)\backslash\mathbb{H}$ having genus $0$.
\end{proof}

\begin{corollary}\label{corintersectionbounds}
Let $q_1,q_2$ be a pair of PIBQFs with river period lengths $p_1,p_2$. Then
\[4\leq\Int(q_1,q_2)\leq p_1p_2.\]
\end{corollary}
\begin{proof}
Each possible intersection came from a pair of vertices on the river modulo the periods, which gives the upper bound. For the lower bound, it suffices to prove that $\Int^{RS}(q_1,q_2)\geq 1$. Since the river sequences contain at least one $0$ and one $1$, we can find the subsequence $01$ in the first river, and $10$ in the second. This will correspond to an intersection of type $RS$, completing the proof.
\end{proof}

\subsection{Explicit computation of the intersection number}

The proof of Proposition \ref{propintsign0} gives us a nice algorithm to compute $\Int^{RS}(q_1,q_2)$. Since $\Int^{RS}(-q_2,q_1)=\Int^{RO}(q_1,q_2)$, Corollary \ref{corintiseven} and two applications of Algorithm \ref{algsl2zintcalcOLD} allow us to compute the full intersection number $\Int(q_1,q_2)$.

\begin{algorithm}\label{algsl2zintcalcOLD}
Given a pair of PIBQFs $q_1,q_2$, this algorithm calculates $\Int^{RS}(q_1,q_2)$.
\begin{enumerate}
\item Compute the continued fraction expansions of $q_{1,f}$ and $q_{2,f}$, and use Remark \ref{remcontinuedfrac} to translate this into the river sequences $r_1=(x_1,x_2,\ldots,x_m)$ and $r_2=(y_1,y_2,\ldots,y_n)$ for $q_1$ and $q_2$ respectively (with $L=0$ and $R=1$).
\item Initialize $i=0$, $j=0$, and $I=0$.
\item Increment $i$ until $x_i=0$ or $i=m+1$. In the first case proceed onward, and in the second case return $I$.
\item Increment $j$ until $y_j=1$ or $j=n+1$. In the first case proceed onward, and in the second set set $j=1$ and go back one step.
\item Set $k=0$, and increment $k$ until $x_{i+k}\neq y_{j+k}$ (indices taken modulo $m,n$ respectively). 
\item If $x_{i+k}=1$, increment $I$ by $1$.
\item Return to step 4.
\end{enumerate}
\end{algorithm}

Note that Algorithm \ref{algsl2zintcalcOLD} will terminate, since for each pair $(i,j)$, step 5 will terminate for some $k\leq mn$.

While Algorithm \ref{algsl2zintcalcOLD} is a natural way to compute the intersection number, there is some loss of efficiency. If the same subsequence occurs multiple times in $r_1$ (which will happen more and more frequently as $m$ grows, due to the pigeonhole principle), then corresponding subsequences in $r_2$ will be compared against it multiple times. As such, it should be more efficient to iteratively compute shared subsequences, keeping track of all places that they occur in each sequence. In small individual cases this can be slower, but on average it becomes significantly faster (at the cost of requiring more memory). A recursive implementation of this is described in Algorithm \ref{algsl2zintcalc}. The implemented code in \cite{Qquad} follows this implementation, but does so non-recursively (the recursive variant provides a simpler exposition, but is less efficient).

If $I$ is a set of indices of an $n-$length sequence, let $I+k$ denote $\{i+k\pmod{n}:i\in I\}$, i.e. cyclically incrementing the indices by $k$.

\begin{algorithm}\label{algsl2zintcalc}
Given a pair of PIBQFs $q_1,q_2$, this algorithm calculates $\Int^{RS}(q_1,q_2)$.
\begin{enumerate}
\item Compute the continued fraction expansions of $q_{1,f}$ and $q_{2,f}$, and use Remark \ref{remcontinuedfrac} to translate this into the river sequences $r_1=(x_1,x_2,\ldots,x_m)$ and $r_2=(y_1,y_2,\ldots,y_n)$ for $q_1$ and $q_2$ respectively (with $L=0$ and $R=1$). At the same time, initialize the sets $i_{1,0},i_{1,1},i_{2,0},i_{2,1}$, where $i_{1,0}$ represents the indices $i$ for which $x_i=0$ (and analogously with the other sets).
\item Initialize $S=\cdot$, $T_1=i_{1,0}+1$, $T_2=i_{2,1}+1$, $k=1$, and $I=0$ to store the common subsequence (as a word of $0$'s and $1$'s), indices of the first river, indices of the second river, the current recursion layer, and the intersection number respectively. Assume that $I$ is globally defined.
\item Given $(S, T_1, T_2, k)$, if $T_1$ or $T_2$ is empty, return to the previous layer $k-1$.
\item Increment $I$ by $|i_{1,1}\cap T_1|\cdot |i_{2,0}\cap T_2|$.
\item Set $S'=S0$, $T_1'=i_{1,0}\cap T_1+1$, and $T_2'=i_{2,0}\cap T_2+1$, and call Step 3 with the input $(S', T_1', T_2', k+1)$.
\item Set $S'=S1$, $T_1'=i_{1,1}\cap T_1+1$, and $T_2'=i_{2,1}\cap T_2+1$, and call Step 3 with the input $(S', T_1', T_2', k+1)$.
\item If we are in the base layer $k=1$ of the recursion, output $I$ and exit the algorithm. Otherwise, return to the previous layer $k-1$.
\end{enumerate}
\end{algorithm}
\begin{proof}
In layer $k$ of the algorithm, the indices $T_1$ are the set of $i$ for which the words $x_{i-k}x_{i-k+1}\cdots x_{i-1}$ and $0S$ are equal, and the indices $T_2$ are the set of $i$ for which the words $y_{i-k}y_{i-k+1}\cdots y_{i-1}$ and $1S$ are equal. In step $4$ we add the intersections corresponding to the sequences $0S1$ in $r_1$ and $1S0$ in $r_2$, and in steps 5, 6 we increase the length of $S$ by $1$ and apply recursion. Since the length of $S$ cannot exceed $mn$, the algorithm will terminate.
\end{proof}

To compare Algorithms \ref{algsl2zintcalcOLD} and \ref{algsl2zintcalc}, we took various discriminant ranges for $D_1, D_2$, took $10000$ random trials of pairs $q_1, q_2$ of PIBQFs with discriminants in the respective ranges, computed their rivers, and recorded the average time taken to compute $\Int^{RS}(q_1, q_2)$, given the precomputed rivers (in the sense of Step 1 of Algorithm \ref{algsl2zintcalc}, where the sets $i_{j,k}$ are also precomputed). The output is in Table \ref{tabinumtimes}.

\begin{table}[ht]
\renewcommand{\arraystretch}{1.2}
\centering
\caption{Average time to compute $\Int^{RS}(q_1,q_2)$}\label{tabinumtimes}
\begin{threeparttable}
\begin{tabular}{|c|c|c|c|c|c|c|c|} 
\hline
$D_1$ range & $D_2$ range & $p_{1,avg}$\tnote{a} & $p_{2,avg}$\tnote{b} & $\Int_{avg}^{RS}$\tnote{c} & T(river)\tnote{d} & T(\ref{algsl2zintcalcOLD})\tnote{e} & T(\ref{algsl2zintcalc})\tnote{f} \\ \hline
$[1,1000]$  & $[1,1000]$                     & $48.4$   & $47.8$   & $93.9$    & $0.016$ & $0.025$  & $0.014$\\ \hline
$[1,1000]$  & $[1001,2000]$                  & $48.7$   & $92.4$   & $161.5$   & $0.022$ & $0.050$  & $0.020$\\ \hline
$[1001,2000]$  & $[10^6+1,10^6+10^3]$        & $91.7$   & $2971.4$ & $4552.8$  & $0.338$ & $2.740$  & $1.089$\\ \hline
$[10^6+1,10^6+10^3]$  & $[10^6+1,10^6+10^3]$ & $3014.0$ & $2941.7$ & $83338.1$ & $0.723$ & $87.553$ & $8.069$\\ \hline
\end{tabular}
\begin{tablenotes}
   \item[a] Average length of the river of $q_1$, rounded to one decimal place.
   \item[b] Average length of the river of $q_2$, rounded to one decimal place.
	 \item[c] Average $RS-$intersection number, rounded to one decimal place.
	 \item[d] Average time to compute the rivers in milliseconds rounded to three decimal places.
	 \item[e] Average time for Algorithm \ref{algsl2zintcalcOLD} in milliseconds rounded to three decimal places.
	 \item[f] Average time for Algorithm \ref{algsl2zintcalc} in milliseconds rounded to three decimal places.
\end{tablenotes}
\end{threeparttable}
\end{table}

As the discriminant (and river lengths) scale up, Algorithm \ref{algsl2zintcalc} quickly takes over in efficiency.

\section{Numerical Calculations}\label{secnumcalc}

Algorithms to compute intersection numbers have been implemented in PARI (\cite{PARI}) as part of the larger package \cite{Qquad}. In this section, we present some computations illustrating the results of this paper.

\subsection{Explicit examples}

In Example \ref{exintis8}, we use Algorithm \ref{algsl2zintcalcOLD} on a nice family of quadratic forms.

\begin{example}\label{exintis8}
If $q_1=[1,1,-1]$ and $q_2=[1,n,-1]$ (for $n\geq 2$), then $\Int(q_1,q_2)=8$. 
\end{example}
\begin{proof}
The automorph of $q_1$ is
\[\left(\begin{matrix}1 & 1\\1&2\end{matrix}\right)=RL,\]
and it can be shown that the automorph of $q_2$ is
\[\left(\begin{matrix} 1 & n\\n & n^2+1\end{matrix}\right)=R^nL^n.\]
Therefore, 
\[r_1=(1,0), \qquad r_2=(1,1,\ldots,1,0,0,\ldots,0).\]
Since $q_1$ is reciprocal, by Corollary \ref{corintiseven}, $\Int(q_1,q_2)=4\Int^{\text{RS}}(q_1,q_2)$. In Algorithm \ref{algsl2zintcalcOLD}, we call step 5 on the pairs $(i,j)=(2,x)$ for $1\leq x\leq n$. If $x\leq n-2$, then $k=2$, corresponding to subsequences $010$ and $111$ respectively, and this gives no intersections. If $x=n-1$, then $k=3$, corresponding to subsequences $0101$ and $1100$, giving an intersection. Finally, if $x=n$, then $k=1$, corresponding to subsequences $01$ and $10$, again giving an intersection. Therefore, $\Int^{\text{RS}}(q_1,q_2)=2$, as claimed.
\end{proof}

Algorithm \ref{algsl2zintcalc} and Theorem \ref{mainthm1}/Corollary \ref{maincor1} provide two ways to compute $\Int(D_1,D_2)$, so it is instructive to verify that they agree.

\begin{example}\label{exgzcheck}
We have
\[\Int(5,136)=48.\]
\end{example}
\begin{proof}
The narrow class groups of discriminants $5, 136$ are $1,\Z/4\Z$ respectively, and have representatives
\begin{align*}
\Cl^+(5)= & \{[1,1,-1]\},\\
\Cl^+(136)= & \{[1,10,-9],[-3,10,3],[9,10,-1],[-5,6,5]\}.
\end{align*}
Using Algorithm \ref{algsl2zintcalc}, we compute
\begin{align*}
\Int([1,1,-1],[-3,10,3])= & \Int([1,1,-1],[-5,6,5])= 16,\\
\Int([1,1,-1],[1,10,-9])= & \Int([1,1,-1],[9,10,-1])= 8,
\end{align*}
hence
\[\Int(5,136)=16+16+8+8=48.\]

Alternatively, using Theorem \ref{mainthm1}, we find that for $n\in S_{5,136}$, we have $p(n)\neq 0$ if and only if $|n|\in\{2,10,14,18,22,26\}$. For these $n$, we have
\[p(\pm 2)=2,\qquad p(\pm 10)=8,\qquad p(\pm 14)=6,\qquad p(\pm 18)=4,\qquad p(\pm 22)=2,\qquad p(\pm 26)=2.\]
These sum to $2(2+8+6+4+2+2)=48=\Int(5,136)$, as expected.
\end{proof}

\subsection{Distribution of intersection points and angles}
Given a pair of PIBQFs $q_1,q_2$, lift the intersection points to $\PSL(2,\Z)$ equivalence classes in the upper half plane. Proposition \ref{propsl2zintersect} implies that they are CM points, with the defining quadratic form given explitictly. The points all have discriminants being a square divisor of a number of the form $x^2-D_1D_2<0$, but what more can be said about them? In \cite{Duke88}, Duke considers the images on the modular curve of Heegner points and modular geodesics coming from fundamental discriminants $D$. With reference to convex regions with piece-wise smooth boundary, he proves that the Heegner points are equidistributed as $D\rightarrow -\infty$, and the modular geodesics are equidistributed as $D\rightarrow\infty$. We would like to formulate similar results for the case of intersecting modular geodesics.

Start by fixing a PIBQF $q_1$, and another PIBQF $q_2$. Let $z\in\ell_{q_1}$ and let $\ell=\dot{\ell}_{z,\gamma_{q_1}z}$. The intersection points on $\tilde{\ell}_{q_1}$ lift uniquely to $\ell$, so we can study the distribution of intersections on $\tilde{\ell}_{q_1}$ by lifting to $\ell$.

\begin{figure}[ht]
	\centering
		\includegraphics[scale=0.7]{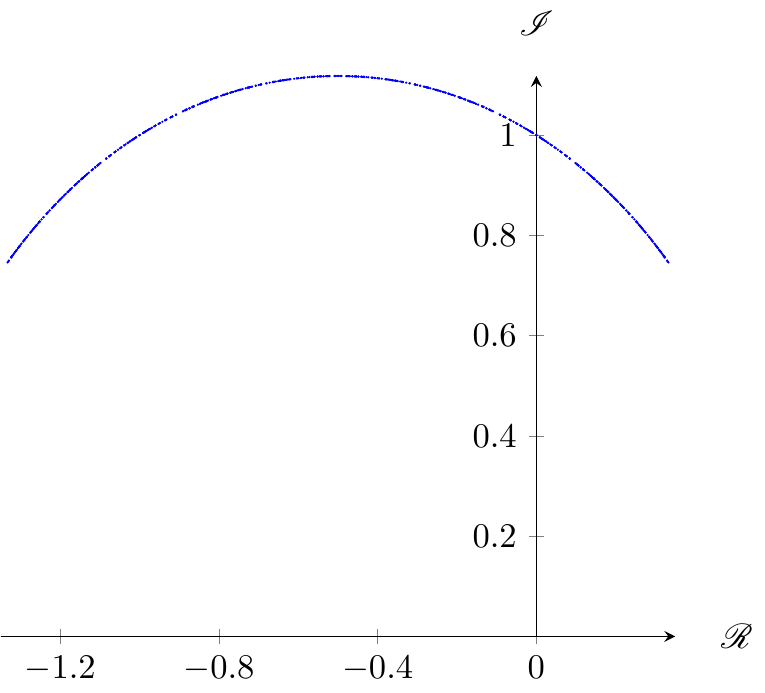}
	\caption{Intersection of $[1,1,-1]$ with discriminant $1000004$.}\label{figintpointscatter}
\end{figure}

A natural guess would be to say that the intersection points become uniformly distributed on $\ell$ as $\disc(q_2)\rightarrow\infty$, and this appears to be true in many examples. However, Example \ref{exintis8} gives a family $q_n=[1,n,-1]$ for which $\disc(q_n)=n^2+4\rightarrow\infty$ and $\Int([1,1,-1],q_n)=8$ for all $n$, which contradicts this. The next reasonable alternative would be to fix $q_1$ and take all forms of discriminant $D$ as $D\rightarrow\infty$. For example, let $q_1=[1,1,-1]$, let $D=1000^2+4=1000004$, and let $z=\frac{-4+\sqrt{5}i}{3}$ (to optimize the symmetry). We find that $h^{+}(D)=52$, there are $1640$ intersection points, and they generate Figure \ref{figintpointscatter}. They seem reasonably well distributed (the imaginary parts are all fairly large, so the effects of the hyperbolic metric are not as obvious), and the ``deficiency'' of intersections between $q_1$ and $[1,1000,-1]$ has been compensated for.

Taking this one step further, let $D=10^{12}+2021$. There are $1467920$ intersections points, and we calculate the hyperbolic distance (along $\tilde{\ell}_{q_1}$) between the image of $z=\frac{-4+\sqrt{5}i}{3}$ and the intersection points. By using $400$ bins, we generate a histogram in Figure \ref{figintpointshist}. The data appears fairly equidistributed, and other examples yield similar results, hence we formalize this statement in a conjecture.

\begin{conjecture}\label{conjintdist}
Let $q$ be fixed, let $D$ be a discriminant, and let $I_{q}(D)$ denote the multiset of points on $\tilde{\ell}_{q}$ that appear as intersections between $q$ and a form of discriminant $D$. Then the set $I_{q}(D)$ is equidistributed (with respect to the hyperbolic metric) on $\tilde{\ell}_{q}$ as $D\rightarrow\infty$.
\end{conjecture}

\begin{figure}[!hb]
\centering
\begin{minipage}{.5\textwidth}
  \centering
  \includegraphics[width=.9\linewidth]{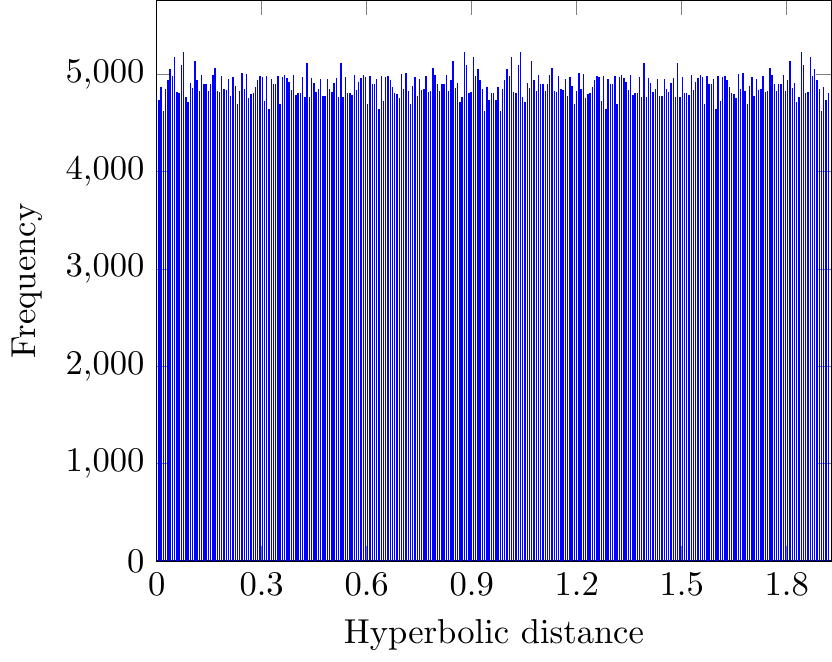}
  \captionof{figure}{\protect\raggedright$1467920$ intersection points.}\label{figintpointshist}
\end{minipage}%
\begin{minipage}{.5\textwidth}
  \centering
  \includegraphics[width=.9\linewidth]{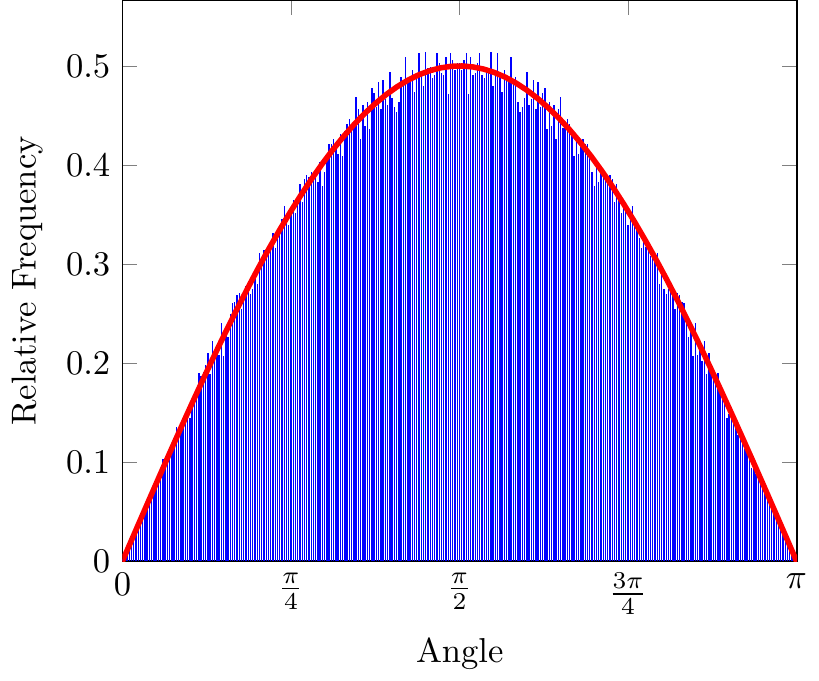}
  \captionof{figure}{$1467920$ intersection angles, scaled.}\label{figintangles}
\end{minipage}
\end{figure}

A similar topic of study would be the distribution of the intersection angles. We take the domain of $\arctan$ to be $[0,\pi)$, and as before, fix $q_1=[1,1,-1]$ and let $D=10^{12}+2021$. By using $300$ bins, we generate a histogram in Figure \ref{figintangles}, where we have scaled the figure to have area $1$.

Small values of $B_{\Delta}$ correspond to angles close to $\frac{\pi}{2}$, and large values correspond to angles close to $0$ (if $B_{\Delta}>0$) or $\pi$ (if $B_{\Delta}<0$). The function $y=\frac{1}{2}\sin(x)$ (which has area $1$ in this range) has been drawn in red in Figure \ref{figintangles}. It matches the data well, and other cases provide similar pictures, hence we formulate this as a conjecture.

\begin{conjecture}\label{conjangdist}
Let $q$ be fixed, let $D$ be a discriminant, and let $\theta_{q}(D)$ denote the multiset of angles that appear as intersections between $q$ and a form of discriminant $D$. As $D\rightarrow\infty$, $\theta_{q}(D)$ tends towards the distribution $\frac{1}{2}\sin(x)$ on $[0,\pi]$.
\end{conjecture}

By looking at $\cos(\theta_{q}(D))$ and recalling Proposition \ref{propsl2zintersect}, Conjecture \ref{conjangdist} is equivalent to showing that $\{\frac{x}{\sqrt{D_qD}}\}$ equidistributes as $D\rightarrow\infty$, where $q$ has discriminant $D_q$, and $x=B_{\Delta}(q,q')$ as $(q,q')$ ranges over the equivalence classes of intersections with $q'$ having discriminant $D$.

\begin{errata}
Conjectures \ref{conjintdist} and \ref{conjangdist} have been resolved by Junehyuk Jung and Naser Sardari in \cite{JS21}!
\end{errata}

\subsection{Distribution of the total intersection number}

Theorem \ref{mainthm1} gives us a formula for $\Int(D_1,D_2)$ when $D_1,D_2$ are coprime and fundamental, but this formula is still somewhat mysterious. For example, it is not even clear that $\Int(D_1,D_2)\neq 0$!

If $D$ is a discriminant, let $R^{+}(D)=\log(T+U\sqrt{D})$ denote the positive regulator associated to $D$ (where $(T,U)$ is the smallest solution to $T^2-DU^2=4$). Since the lengths of closed geodesics corresponding to $D_i$ are $2R^+(D_i)$, it would be reasonable to expect their intersection number to be proportional to $R^+(D_1)R^+(D_2)$. As such, define
\[C_{D_1,D_2}:=\dfrac{\Int(D_1,D_2)}{h^{+}(D_1)h^{+}(D_2)R^{+}(D_1)R^{+}(D_2)}\]
be the average linking of forms of discriminant $D_1,D_2$ divided by $R^+(D_1)R^+(D_2)$. 

We took $100000$ trials with $1\leq D_1\leq 10^4$ and $1\leq D_2\leq 10^6$, and computed $C_{D_1,D_2}$ for each pair. By using $1170$ bins, we produce the histogram Figure \ref{figCD1D2}. There are $49$ data points lying outside of the displayed area.

\begin{figure}[!ht]
	\centering
		\includegraphics[scale=0.8]{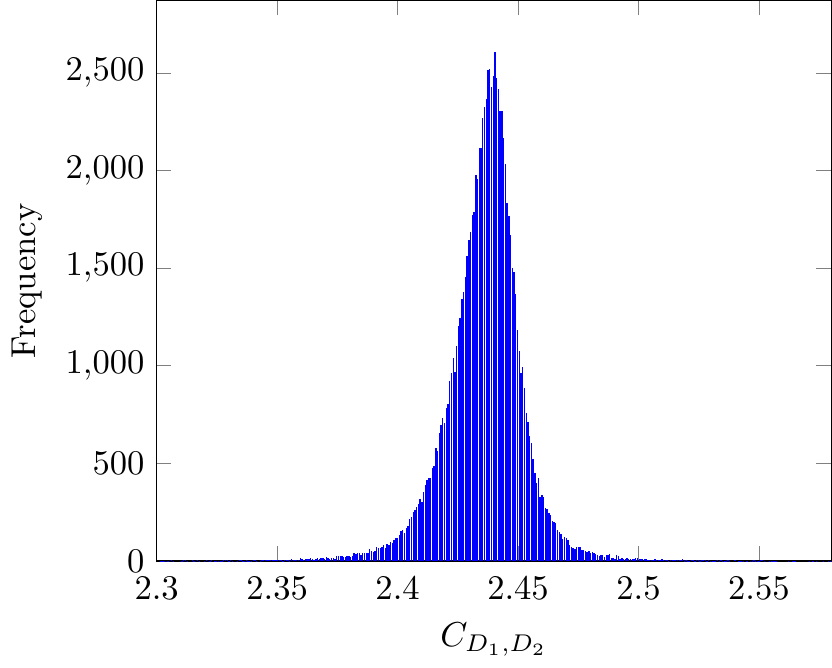}
	\caption{$100000$ trials of $C_{D_1,D_2}$.}\label{figCD1D2}
\end{figure}

The statistics of the data are found in Table \ref{tabcd1d2stats} (all data rounded to 5 decimal places).

\begin{table}[ht]
\centering
\caption{$C_{D_1, D_2}$ for $100000$ random trials.}\label{tabcd1d2stats}
\begin{tabular}{|c|c|c|c|c|c|c|} 
\hline
Minimum & $25^{\text{th}}$ percentile & Median & $75^{\text{th}}$ percentile & Maximum & Average & Standard Deviation\\ \hline
2.17485 & $2.42854$                & $2.43776$ & $2.44528$               & $3.01915$ & $2.43653$ & $0.01727$\\ \hline
\end{tabular}
\end{table}

The data suggests that $C_{D_1, D_2}$ is bounded, and typically very close to $2.437$, which supports the heuristic.

\bibliographystyle{alpha}
\bibliography{../../references}

\begin{thebibliography}{{The}21}

\bibitem[Bir17]{BK17}
George~D. Birkhoff.
\newblock Dynamical systems with two degrees of freedom.
\newblock {\em Trans. Amer. Math. Soc.}, 18(2):199--300, 1917.

\bibitem[DIT17]{DIT17}
W.~Duke, \"{O}. Imamo\={g}lu, and \'{A}. T\'{o}th.
\newblock Modular cocycles and linking numbers.
\newblock {\em Duke Math. J.}, 166(6):1179--1210, 2017.

\bibitem[Duk88]{Duke88}
W.~Duke.
\newblock Hyperbolic distribution problems and half-integral weight {M}aass
  forms.
\newblock {\em Invent. Math.}, 92(1):73--90, 1988.

\bibitem[DV20]{DV20}
Henri Darmon and Jan Vonk.
\newblock Singular moduli for real quadratic fields: A rigid analytic approach.
\newblock {\em Duke Math. J.}, 2020.

\bibitem[Ghy07]{Ghys07}
\'{E}tienne Ghys.
\newblock Knots and dynamics.
\newblock In {\em International {C}ongress of {M}athematicians. {V}ol. {I}},
  pages 247--277. Eur. Math. Soc., Z\"{u}rich, 2007.

\bibitem[GZ85]{GZ85}
Benedict~H. Gross and Don~B. Zagier.
\newblock On singular moduli.
\newblock {\em J. Reine Angew. Math.}, 355:191--220, 1985.

\bibitem[JS21]{JS21}
Junehyuk Jung and Naser~Talebizadeh Sardari.
\newblock Intersecting geodesics on the modular surface, 2021.

\bibitem[Mil71]{Milnor71}
John Milnor.
\newblock {\em Introduction to algebraic {$K$}-theory}.
\newblock Princeton University Press, Princeton, N.J.; University of Tokyo
  Press, Tokyo, 1971.
\newblock Annals of Mathematics Studies, No. 72.

\bibitem[Ric21a]{JR21shim}
James Rickards.
\newblock Counting intersection numbers of closed geodesics on {S}himura
  curves, 2021.

\bibitem[Ric21b]{Qquad}
James Rickards.
\newblock Q- {Quadratic}.
\newblock \url{https://github.com/JamesRickards-Canada/Q-Quadratic}, 2021.

\bibitem[SV18]{SV18}
K.~Spalding and A.~P. Veselov.
\newblock Conway river and {A}rnold {S}ail.
\newblock {\em Arnold Math. J.}, 4(2):169--177, 2018.

\bibitem[{The}21]{PARI}
{The PARI~Group}, Univ. Bordeaux.
\newblock {\em {PARI/GP version \texttt{2.13.2}}}, 2021.
\newblock available from \url{http://pari.math.u-bordeaux.fr/}.

\end{thebibliography}

\end{document}